\documentclass[12pt]{amsart}
\usepackage{amsmath,amssymb,url,color,tikz,colortbl}
\usetikzlibrary{patterns}
\usepackage[all]{xy}
\usepackage{young}
\usepackage{enumerate}
\usepackage[margin=1in]{geometry}
\usepackage{hyperref}
\usepackage{amsfonts}
\usepackage{amsmath}
\usepackage{amssymb}
\usepackage{amsthm}
\usepackage{xcolor}

\usepackage{graphicx}
\usepackage{listings}
\usepackage{tram}
\usepackage{manfnt}
\usepackage{url,color,tikz,euscript}

\usepackage[colorinlistoftodos,color=red!70]{todonotes}

\theoremstyle{theorem}
\newtheorem{thm}{Theorem}
\newtheorem{theorem}[thm]{Theorem}

\newtheorem{lemma}[thm]{Lemma}
\newtheorem{prop}[thm]{Proposition}
\newtheorem{proposition}[thm]{Proposition}
\newtheorem{cor}[thm]{Corollary}

\theoremstyle{definition}
\newtheorem{defn}[thm]{Definition}
\newtheorem{definition}[thm]{Definition}
\newtheorem{ex}[thm]{Example}
\newtheorem{example}[thm]{Example}

\newtheorem{remark}[thm]{Remark}

\numberwithin{thm}{section}
\catcode`@=11
\def\m@th{\mathsurround\z@}
\def\cases#1{\left\{\,\vcenter{\normalbaselines\m@th
    \ialign{$##\hfil$&\quad##\hfil\crcr#1\crcr}}\right.}
\def\hang{\hangindent 24pt}
\def\d@nger{\medbreak\begingroup\clubpenalty=10000
  \def\par{\endgraf\endgroup\medbreak} %
  \noindent\hang\hangafter=-2
  \hbox to0pt{\hskip-\hangindent\dbend\hfill}}
\outer\def\danger{\d@nger}

\newcommand{\zz}{\mathbb{Z}}

\newcommand{\kk}{\mathbb{K}}
\renewcommand{\ss}{\mathfrak{S}}

\newcommand{\pro}{\operatorname{Pro}}
\newcommand{\rev}{\operatorname{rev}}
\newcommand{\tog}{\operatorname{Tog}}
\newcommand{\row}{\operatorname{Row}}

\newcommand{\fix}{\operatorname{Fix}}

\newcommand{\ra}{\rightarrow}
\newcommand{\sm}{\setminus}
\definecolor{green}{HTML}{006600}
\definecolor{orange}{HTML}{FF6200}
\definecolor{purple}{HTML}{990099}

\catcode`@=12 

\newcommand{\eso}{\EuScript{O}}

\newcommand{\cali}{\mathcal{I}}

\newcommand{\call}{\mathcal{L}}

\newcommand{\calp}{\mathcal{P}}

\newcommand{\cals}{\mathcal{S}}
\newcommand{\calt}{\mathcal{T}}

\newcommand{\calz}{\mathcal{Z}}

\begin{document}

\title{Toggling Independent Sets of a Path Graph}
\author[Joseph]{Michael Joseph}
\address{Department of Technology and Mathematics, Dalton State College, 650 College Dr., Dalton, GA 30720, USA}
\email{mjjoseph@daltonstate.edu}
\author[Roby]{Tom Roby}
\address{Department of Mathematics, University of Connecticut, Storrs, CT 06269-1009, USA}
\email{tom.roby@uconn.edu}
\date{\today}


\begin{abstract}
This paper explores the orbit structure and homomesy (constant averages over orbits)
properties of certain actions of toggle groups on the collection of independent sets of a
path graph.  In particular we prove a generalization of a homomesy conjecture of Propp
that for the action of a ``Coxeter element'' of vertex toggles,
the difference of indicator functions of symmetrically-located vertices is 0-mesic.
Then we use our analysis to show facts about orbit sizes that are easy to conjecture but
nontrivial to prove.
Besides its intrinsic interest, this particular combinatorial dynamical system is valuable in
providing an interesting example of (a) homomesy in a context where large orbit sizes make a
cyclic sieving phenomenon unlikely to exist, (b) the use of Coxeter theory to greatly
generalize the set of actions for which results hold, and (c) the usefulness of Striker's
notion of generalized toggle groups.\vspace{1 ex}

\end{abstract}

\maketitle
\vspace{-0.333 in}
\section{Introduction}
The past several years have seen renewed interest in the area of \emph{dynamical algebraic
combinatorics}, where one considers group actions on sets of discrete combinatorial objects
and looks for interesting properties of their orbits or orders.  These include such actions as
\emph{promotion} (on Young tableaux or posets) and \emph{rowmotion} of posets, but there are
many others.  Many of these can be built up as a sequence of simple involutions, as is the
main action we study here on the collection of independent sets of a path graph.  

Our goal is to understand the orbit structure and \emph{homomesy} (constant averages over
orbits) of this discrete dynamical system.  In particular we prove a conjecture of Propp
that, with respect to the action ``toggle once at each vertex from left to right'', the
difference of indicator functions of symmetrically-located vertices is 0-mesic
(Theorem~\ref{indepsethomomesy}).  By leveraging some basic theory of Coxeter groups, we
generalize this result to apply to \emph{any} Coxeter element in the toggle group.  Our
analysis allows us to deduce facts about orbit sizes that are apparent from the numerical
data, but seem difficult to prove directly.  

This particular combinatorial dynamical system serves as a valuable case study in
several respects.  First, it provides an interesting example of the homomesy phenomenon in a context where
unwieldy orbit sizes suggest the lack of a natural cyclic sieving phenomenon (CSP) in the sense of Reiner, Stanton, and
White~\cite{csp}.  Many combinatorial dynamical systems that support a CSP also
have natural homomesic statistics and vice versa, though there appears to be no direct
connection between the two (even in specific cases). 

Second, by taking a Coxeter theoretic approach, we are able to greatly generalize the set of
actions for which our results hold, from the specific action of successively toggling at each vertex
to toggling once per vertex in an arbitrary order (Subsection~\ref{subsec:Coxeter}).

Third, this combinatorial dynamical system displays the usefulness of Striker's notion of generalized toggle
groups~\cite{strikergentog} to settings beyond that of posets. Although there is an
equivariant bijection (Proposition~\ref{correspond}) between the action we study on
independent sets and the action of promotion on zigzag posets, it is much
easier to establish the homomesy in the former setting first, then translate it to the
latter setting.

We now describe the setting and background necessary to understand the problem.

\begin{defn} Let $\calp_n$ denote the \textbf{path graph} with vertex set $[n]:=\{1,2,\dots
,n\}$ and edge set $\{ \{i,i+1 \}: i\in [n-1] \}$.
\end{defn}

\begin{ex}
The path graph with seven vertices is
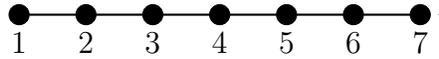
\begin{center}\begin{tikzpicture}[scale=8/9]
\draw[thick, -] (1.15,0) -- (1.85,0);
\draw[thick, -] (2.15,0) -- (2.85,0);
\draw[thick, -] (3.15,0) -- (3.85,0);
\draw[thick, -] (4.15,0) -- (4.85,0);
\draw[thick, -] (5.15,0) -- (5.85,0);
\draw[thick, -] (6.15,0) -- (6.85,0);
\draw[fill] (1,0) circle [radius=0.15];
\draw[fill] (2,0) circle [radius=0.15];
\draw[fill] (3,0) circle [radius=0.15];
\draw[fill] (4,0) circle [radius=0.15];
\draw[fill] (5,0) circle [radius=0.15];
\draw[fill] (6,0) circle [radius=0.15];
\draw[fill] (7,0) circle [radius=0.15];
\node[below] at (1,-0.1) {1};
\node[below] at (2,-0.1) {2};
\node[below] at (3,-0.1) {3};
\node[below] at (4,-0.1) {4};
\node[below] at (5,-0.1) {5};
\node[below] at (6,-0.1) {6};
\node[below] at (7,-0.1) {7};
\node at (7.3,0) {.};
\end{tikzpicture}
\end{center}
\end{ex}

\begin{defn}
An \textbf{independent set} of a graph is a subset of the vertices that does not contain a pair of adjacent vertices.  Let $\cali_n$ denote the set of independent sets of $\calp_n$.
\end{defn}

\begin{ex}
The set of vertices $\{1,4,6\}$ represented
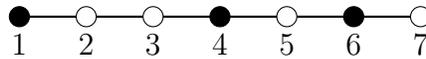
\begin{center}\begin{tikzpicture}[scale=8/9]
\draw[thick, -] (1.15,0) -- (1.85,0);
\draw[thick, -] (2.15,0) -- (2.85,0);
\draw[thick, -] (3.15,0) -- (3.85,0);
\draw[thick, -] (4.15,0) -- (4.85,0);
\draw[thick, -] (5.15,0) -- (5.85,0);
\draw[thick, -] (6.15,0) -- (6.85,0);
\draw[fill] (1,0) circle [radius=0.15];
\draw (2,0) circle [radius=0.15];
\draw (3,0) circle [radius=0.15];
\draw[fill] (4,0) circle [radius=0.15];
\draw (5,0) circle [radius=0.15];
\draw[fill] (6,0) circle [radius=0.15];
\draw (7,0) circle [radius=0.15];
\node[below] at (1,-0.1) {1};
\node[below] at (2,-0.1) {2};
\node[below] at (3,-0.1) {3};
\node[below] at (4,-0.1) {4};
\node[below] at (5,-0.1) {5};
\node[below] at (6,-0.1) {6};
\node[below] at (7,-0.1) {7};
\end{tikzpicture}\end{center} is an independent set of $\calp_7$, but $\{1,4,5,6\}$ represented 
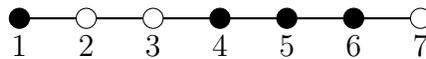
\begin{center}\begin{tikzpicture}[scale=8/9]
\draw[thick, -] (1.15,0) -- (1.85,0);
\draw[thick, -] (2.15,0) -- (2.85,0);
\draw[thick, -] (3.15,0) -- (3.85,0);
\draw[thick, -] (4.15,0) -- (4.85,0);
\draw[thick, -] (5.15,0) -- (5.85,0);
\draw[thick, -] (6.15,0) -- (6.85,0);
\draw[fill] (1,0) circle [radius=0.15];
\draw (2,0) circle [radius=0.15];
\draw (3,0) circle [radius=0.15];
\draw[fill] (4,0) circle [radius=0.15];
\draw[fill] (5,0) circle [radius=0.15];
\draw[fill] (6,0) circle [radius=0.15];
\draw (7,0) circle [radius=0.15];
\node[below] at (1,-0.1) {1};
\node[below] at (2,-0.1) {2};
\node[below] at (3,-0.1) {3};
\node[below] at (4,-0.1) {4};
\node[below] at (5,-0.1) {5};
\node[below] at (6,-0.1) {6};
\node[below] at (7,-0.1) {7};
\end{tikzpicture}\end{center}
is not.  In both of these examples, hollow dots refer to vertices of $\calp_7$ not in the subset.
\end{ex}


Although we sometimes write independent sets as subsets of $[n]:=\{1,2,\dots,n\}$ as above,
it may not be obvious in that notation what the underlying value of $n$ is.  Another notation that is often more convenient for an independent set is its \textbf{binary representation}, in which the bit in position $i$ of $S$ is 0 if $i\not\in S$ and 1 if $i\in S$.  For example 0010010 represents the
independent set $\{3,6\}$ of $\calp_7$.  Thus $\cali_n$ can be viewed as the set of length $n$ binary
strings that do not contain the subsequence 11 (which would indicate the inclusion of
two adjacent vertices). It is well-known and easy to verify that the cardinality of $\cali_{n}$
is a Fibonacci number.

In Section~\ref{sec:togglemain}, we introduce the \emph{toggle group} $\calt_n$ on
$\cali_n$, which is generated by basic involutions called \emph{toggles}.  The notion of
toggle groups goes back
to work of Cameron and Fon-der-Flaass, who introduced toggles on \emph{order ideals} of a
poset~\cite{cameronfonder}.  More recently, Striker
has studied toggle groups in a more general setting~\cite{strikergentog}.
Specifically, given a ``ground" set $X$ and a fixed set of ``allowed" subsets $\call\subseteq 2^X$,
each element $x\in X$ has an associated toggle which removes or inserts $x$ into any
given set in $\call$ provided the resulting set is still in $\call$, and otherwise does
nothing.  In our situation the ground set is $[n]$ and the set of allowed subsets of $[n]$ is $\cali_n$.

Our main results are a proof of a conjecture of Propp (Theorem~\ref{indepsethomomesy}) and a generalization of it (Corollary~\ref{cor:hom-Coxeter}).  These theorems give examples of the \emph{homomesy} (Greek
for ``same middle'') \emph{phenomenon}, introduced by Propp and the second author in
\cite{propproby} and defined as follows.

\begin{defn}
Suppose we have a set $\cals$, an invertible map $w:\cals\ra \cals$ such that every $w$-orbit is finite, and a function (``statistic") $f:\cals\ra\kk$, where $\kk$ is a
field of characteristic 0.  Then we say the triple $(\cals,w,f)$ exhibits \textbf{homomesy}
if there exists a constant $c\in\kk$ such that for every $w$-orbit $\eso\subseteq \cals$, $$\frac{1}{\#\eso}\sum\limits_{x\in\eso}f(x)=c.$$  In this case, we say that the function $f$ is \textbf{homomesic with average} $\mathbf{\textsl{c}}$, or $\textbf{\textsl{c}-mesic}$, under the action of $w$ on $\cals$.
\end{defn}

Some early isolated examples of homomesy exist in the literature, notably a conjecture of
Panyushev \cite[Conjecture 2.1(iii)]{panyushev} for the rowmotion operator acting on
antichains in positive root posets.  This was proven by Armstrong, Stump, and
Thomas~\cite{ast}, but investigation of homomesy as a widespread phenomenon is more recent.  Examples now
include cyclic actions on partitions, Suter's action on Young diagrams, rowmotion and
promotion of order ideals, Lyness 5-cycles (which has strong connections to cluster algebra theory), and certain toggling actions for noncrossing
partitions~\cite{einpropp,efgjmpr,shahrzad,propproby,robydac,strikerwilliams,strikerRS}.

Although it is a new area of research, the homomesy phenomenon has been discovered in a wide
variety of combinatorial dynamical systems.  Most of the initially proven homomesy results were
for systems in which the order of the map was known.  In fact, in many cyclic actions where
homomesy is present, one also finds the cyclic sieving phenomenon of Reiner, Stanton, and
White~\cite{csp}.  (See also Sagan's more leisurely exposition of the basic ideas and
examples~\cite{cspsagan}.)  There have long been conjectured
homomesies for maps with unpredictable orbit sizes, but the first such proven result came
out of a team (including the first author) assembled at an AIM workshop~\cite{efgjmpr}.  In
this paper, we eventually discuss how to determine the sizes of the orbits under our maps, but they do not
divide a number that is easy to describe without listing all of them.


To prove Propp's original conjecture (Theorem~\ref{indepsethomomesy}), we associate an
\emph{orbit board} to each orbit, and partition the ones in the orbit board into \emph{snakes}
which begin in the left column and end in the right column.  Our technique was inspired by
Haddadan's proof that the ``winching'' action on $k$-element subsets of
$[n]$ exhibits homomesy (another conjecture of Propp)~\cite{shahrzad}.
In addition to proving homomesy, the snake representations lead to many other results, e.g.,
on the total number of orbits or the existence of orbits of certain sizes, the latter being
the focus of Section~\ref{sec:orbitsizes}.  We first consider a specific action
$\varphi\in\calt_n$ which toggles left to right, but most of our results can be generalized
to certain other actions in $\calt_n$ named \emph{Coxeter elements}.  This is the focus of
Subsection~\ref{subsec:Coxeter}.  We use some theory of Coxeter groups to explain why we can extend proven results for $\varphi$ to other actions, and this is why we do not start with the more general results.

In Section~\ref{sec:zigzagposets}, we explain how our results can be restated in terms
of toggling order ideals of zigzag posets via an equivariant bijection
$\eta:\cali_n\ra J(\calz_n)$, where $J(\calz_n)$ is the set of order ideals of the zigzag
poset $\calz_n$.  We describe the connections with the well-studied \emph{promotion}
and \emph{rowmotion} operators on the set of order ideals of a poset~\cite{strikerwilliams}.
However, the proofs of our main results are much easier to obtain by working with
toggles on $\cali_n$ as opposed to those
on $J(\calz_n)$, which shows the significance of considering toggle groups in Striker's
generalized setting, instead of the original setting of toggling order ideals from Cameron
and Fon-der-Flaass.  In fact, the independent sets of $\calp_n$ are the antichains of $\calz_n$ in disguise.  In~\cite[\S3.3, 3.6]{strikergentog}, Striker discusses toggles on antichains of posets and on independent sets of graphs.  In~\cite{antichain-toggling}, The first author describes an explicit isomorphism between the toggle groups of antichains and of order ideals for a general poset.

\subsection{Acknowledgements}\label{ss:ack}

The authors are grateful to James Propp for suggesting this problem initially and helping us
understand its broader context.  We have benefitted from useful discussions with David Einstein, Max Glick, Darij Grinberg, Shahrzad Haddadan, Matthew Macauley, Vic Reiner, Elizabeth Sheridan Rossi, Richard Stanley, Jessica Striker, and Nathan Williams.  Computations
leading to the initial conjectures were done in Sage~\cite{sage}. The Online Encylopedia of
Integer Sequences~\cite{oeis} was invaluable for connecting our data with previously known
sequences. We also thank an anonymous referee for comments that helpd us improve the exposition of
this paper.

\section{Toggle Maps on Independent Sets}\label{sec:togglemain}

In this section we state and prove our main homomesy results.  Throughout this paper we assume $n\geq2$.  While some of our results also hold for $n=1$, many do not, and we are not concerned with this trivial case.

\subsection{Definitions and main results}

We now define the toggles on $\cali_n$.

\begin{defn}
For every $i\in[n]$, define $\tau_i:\cali_n\ra\cali_n$, the \textbf{toggle at vertex \textsl{i}}, in the following way.  If $i\in S$, $\tau_i$ removes $i$ from $S$, which still results in an independent set.  If $i\not\in S$, then $\tau_i(S)$ adds $i$ to $S$ assuming the resulting set is still independent, and otherwise does nothing.  Formally, $$\tau_i(S)=\left\{\begin{array}{ll}
S\sm\{i\}&\text{if }i\in S,\\
S\cup\{i\}&\text{if }i\not\in S\text{ and }S\cup\{i\}\in\cali_n,\\
S&\text{if }i\not\in S\text{ and }S\cup\{i\}\not\in\cali_n.\end{array}\right.$$
\end{defn}
Since the $\tau_{i}$ operate to the left of their arguments, we use the standard convention
that a product of toggles is performed from right to left. It is clear that each $\tau_i$ is
an involution, i.e., $\tau_i^2$ is the identity.  We characterize the order of products of
two toggles in the following propositions.

\begin{prop}\label{prop:togglescommute}
The toggles $\tau_i$ and $\tau_j$ commute if and only if $|i-j|\not=1$.
\end{prop}

\begin{proof}
If $i=j$, then $\tau_i$ and $\tau_j$ clearly commute.

Suppose $|i-j|>1$.  Then whether or not $i$ is in an independent set has no effect on whether or not $j$ can be in that set and vice versa.  So $\tau_i\tau_j=\tau_j\tau_i$.

Suppose $|i-j|=1$.  Then $\tau_i(\tau_j(\O))=\{j\}$ and $\tau_j(\tau_i(\O))=\{i\}$, so $\tau_i\tau_j\not=\tau_j\tau_i$.
\end{proof}

\begin{prop}\label{prop:order6}
When $n\geq3$, the order of the map $\tau_i\circ\tau_j$ is
$$\left\{\begin{array}{ll}
1 & \text{if } i=j,\\
2 & \text{if } |i-j| \geq 2,\\
6 & \text{if } |i-j|=1.\\
\end{array}
\right.$$
\end{prop}

\begin{proof}


The proofs of the first two cases are straightforward since toggles are involutions.

Suppose $|i-j|=1$. Since $\tau_i$ and $\tau_j$ are involutions, $(\tau_i\tau_j)^{-1}=\tau_j\tau_i$, so $\tau_i\tau_j$ has the same order as $\tau_j\tau_i$. Thus, we may assume without loss of generality that $i<j$ (so $j=i+1$).  To show that the order of $\tau_i\tau_{i+1}$ is 6, we will show that there is an orbit of size 2, an orbit of size 3, and no orbit with size greater than 3.

Note that the toggles $\tau_i$ and $\tau_{i+1}$ can only affect whether $i$ and/or $i+1$ are in a given set, and no independent set can contain both. Thus, every orbit under the action of $\tau_i\tau_{i+1}$ can at the very most contain $S$, $S\cup\{i\}$, $S\cup\{i+1\}$, for some $S\in\cali_n$. Therefore all orbits have size at most 3.

The orbit $(\O, \{i+1\}, \{i\})$ has size 3. For $i\geq 2$, the orbit $(\{i-1\},\{i-1,i+1\})$ has size 2. If $i=1$, the orbit $(\{3\},\{1,3\})$ has size 2.  (This is why we needed $n\geq3$, as the map $\tau_1\tau_2$ on $\cali_2$ has order 3, not 6.)


\end{proof}

\begin{definition}
Let $\ss_{\cali_n}$ denote the symmetric group on $\cali_n$. The \textbf{toggle group} of $\cali_n$, denoted $\calt_n$, is the subgroup of $\ss_{\cali_n}$ generated by the $\tau_i$ toggles.
\end{definition}

\begin{defn}
A particular element in $\calt_n$ is $\varphi:=\tau_n\cdots\tau_2\tau_1$, the map that toggles at each vertex from left to right.\end{defn}

\begin{example} In $\cali_5$, $\varphi(10010)=01001$ by the following steps:
$$
10010\stackrel{\tau_{1}}{\longmapsto }
00010\stackrel{\tau_{2}}{\longmapsto }
01010\stackrel{\tau_{3}}{\longmapsto }
01010\stackrel{\tau_{4}}{\longmapsto }
01000\stackrel{\tau_{5}}{\longmapsto }
01001.
$$
\end{example}

Note that $\varphi^{-1}=\tau_1\tau_2\cdots\tau_n$, which applies the toggles right to left.

\begin{defn}Given a set $S\in\cali_n$ and $j\in[n]$, define $\chi_j(S)$ to be the indicator function of vertex $j$ in $S$.  That is, $\chi_j(S)$ is the $j^\text{th}$ digit of the binary representation of $S$.\end{defn}

\begin{example}
$\chi_1(10010)=1$, $\chi_2(10010)=0$, $\chi_3(10010)=0$, $\chi_4(10010)=1$, $\chi_5(10010)=0$.
\end{example}

One of our main theorems to be proven later is the following conjecture of Propp.  We
will later extend this result to actions in $\calt_n$ that are ``Coxeter elements'', i.e.,
products of every $\tau_i$ exactly once in some order (Corollary~\ref{cor:hom-Coxeter}).

\begin{thm}[Propp's conjecture]\label{indepsethomomesy} Under the action of $\varphi$ on
$\cali_n$, $\chi_j-\chi_{n+1-j}$ is 0-mesic for every $1\leq j\leq n$.\end{thm}

\begin{defn}
Given an independent set $S\in\cali_n$ and $w\in\calt_n$, we define the \textbf{orbit board}
for $S$ and $w$ as follows.  Let $S^i=w^i(S)$ for $i\in \zz$ and for any $j\in[n]$, let $S(i,j)=1$ if $j\in
S^i$ and $S(i,j)=0$ if $j\not\in S^i$. (In particular, if $j<1$ or $j>n$, then
$S(i,j)=0$. These are ``out-of-bounds" positions not shown when we display the orbit board.)
\end{defn}

\begin{ex}\label{7ex}
The orbit board for the orbit containing $S=1010100\in\cali_7$ under the action of $\varphi$
is shown in Figure~\ref{fig:7ex}. This is an orbit of size 10, so
$S^{10}=\varphi^{10}(S)=S$. Technically, the orbit board is vertically infinite but
periodic, so we only
show $S^0,S^1,\dots,S^9$ and view it as living on a cylinder.  The element in row $i$ and
column $j$ is $S(i,j)$, with $i\in [0,\ell -1]$ and $j\in [n]$, where $\ell$ is the length
of $S$'s orbit.
Notice that the column-sum vector, $(4,2,3,2,3,2,4)$, is palindromic, illustrating
Theorem~\ref{indepsethomomesy},  since $\chi_j-\chi_{n+1-j}$ has total 0 (and thus average 0)
across this orbit for each $j$.
\end{ex}

\begin{figure}
\begin{center}
\begin{tabular}{c||c|c|c|c|c|c|c}
&\textbf{1}&\textbf{2}&\textbf{3}&\textbf{4}&\textbf{5}&\textbf{6}&\textbf{7}\\\hline\hline\\[-1.111em]
$S^0$&1&0&1&0&1&0&0\\\hline\\[-1.111em]
$S^1$&0&0&0&0&0&1&0\\\hline\\[-1.111em]
$S^2$&1&0&1&0&0&0&1\\\hline\\[-1.111em]
$S^3$&0&0&0&1&0&0&0\\\hline\\[-1.111em]
$S^4$&1&0&0&0&1&0&1\\\hline\\[-1.111em]
$S^5$&0&1&0&0&0&0&0\\\hline\\[-1.111em]
$S^6$&0&0&1&0&1&0&1\\\hline\\[-1.111em]
$S^7$&1&0&0&0&0&0&0\\\hline\\[-1.111em]
$S^8$&0&1&0&1&0&1&0\\\hline\\[-1.111em]
$S^9$&0&0&0&0&0&0&1\\\hline\hline\\[-1.111em]
\textbf{Total} & \textbf{4} & \textbf{2} & \textbf{3} & \textbf{2} & \textbf{3} & \textbf{2} & \textbf{4}
\end{tabular}
\end{center}
\caption{The orbit under the action $\varphi$ on $\cali_7$ containing $S=1010100$.}
\label{fig:7ex}
\end{figure}

A homomesy result which is much simpler to prove is the following.

\begin{thm}\label{babyhomomesy}For $n\geq2$, under the action of $\varphi$ on $\cali_n$, the statistics $2\chi_1+\chi_2$ and $\chi_{n-1}+2\chi_n$ are both 1-mesic.\end{thm}
The reader can easily check that this holds for the orbit in Figure~\ref{fig:7ex}.  This
result can be obtained as a corollary of \cite[Theorem 7.5]{efgjmpr}, but we include another proof here.
As with Theorem~\ref{indepsethomomesy}, we will generalize this theorem to actions by
general Coxeter elements of toggles in Subsection \ref{subsec:Coxeter}.

\begin{proof}We prove that $2\chi_1+\chi_2$ is 1-mesic, as the proof for $\chi_{n-1}+2\chi_n$ is analogous.

The first two bits of any independent set $S$ are either 10, 01, or 00.

If $S$ begins with 10, then when applying $\varphi$ to $S$, the first toggle $\tau_1$
removes the first vertex so the first digit is 0.  Then $\tau_2$ can sometimes insert the
second vertex and sometimes cannot, depending on whether $3\in S$. Thus, $\varphi(S)$ begins
either with 01 or 00.

If $S$ begins with 01, then when applying $\varphi$ to $S$, we leave the first vertex out
and then remove the second vertex.  So $\varphi(S)$ begins with 00.

If $S$ begins with 00, then when applying $\varphi$ to $S$, we insert the first vertex and
then leave the second vertex out. So $\varphi(S)$ begins with 10.

Thus, when repeatedly applying $\varphi$, the first two digits are partitioned into cyclic
patterns of $10 \ra 01 \ra 00$ or $10 \ra 00$.  (An orbit may contain both types of
patterns.) As $2\chi_1+\chi_2$ has average 1 across both types of patterns, it will across
every orbit as well.
\end{proof}

\subsection{Proof of Propp's original conjecture}
Our next goal is to prove Theorem~\ref{indepsethomomesy} via a partitioning of the orbit
board into ``snakes''.  We first note what happens in the special case where symmetry of
independent sets under reversal makes the result obvious.

\begin{defn}The \textbf{reverse} of a word is that word written in the reverse order.  For example, the reverse of 101000010 is 010000101.  Denote the reverse of an independent set $S$ as $S^{\rev}$.  Writing $S$ as a set, $$S^{\rev}=\{n+1-i|i\in S\}.$$\end{defn}

The following is clear because the inverse of $\varphi$ is the function that composes toggling in the reverse order.

\begin{prop}\label{reversal} For any $S$, $(S^{\rev})^{\rev}=S$ and $\varphi(S^{\rev})=\left(\varphi^{-1}(S)\right)^{\rev}$.\end{prop}

\begin{defn} A \textbf{symmetrical} independent set is one that is its own reverse.  For example, 010010 is symmetrical.
\end{defn}


\begin{defn}
A $\varphi$-orbit $\eso$ is \textbf{reversible} if for some $S\in\eso$, $S^{\rev}$ is also in $\eso$.
\end{defn}

For a reversible orbit $\eso$, such as in Example~\ref{7ex}, there exists one $S\in\eso$ whose reverse is also in $\eso$.  Then by Proposition~\ref{reversal}, every set in $\eso$ has its reverse in $\eso$.  So it is clear that $\chi_j-\chi_{n+1-j}$
has average zero across any reversible orbit.  For $n\geq10$, however, there are $\varphi$-orbits on $\cali_n$
that are not reversible, so it is surprising a priori that Theorem~\ref{indepsethomomesy}
holds in general.

\begin{prop}\label{prop:atmost2symm}
Any $\varphi$-orbit containing a symmetrical independent set is reversible.  An orbit contains at most
two symmetrical independent sets (but even a reversible orbit may not contain any).
\end{prop}

\begin{proof}
For an orbit $\eso$ containing a symmetrical independent set $S$, $S^{\rev}=S$ is in the orbit, so $\eso$ is reversible.

Now assume that an orbit $\eso$ contains at least two different symmetrical independent sets $S$ and $T$.  Then there exists $m\geq1$ such that $\varphi^m(S)=T$ and let $m$ be the least number that satisfies this.  Then from Proposition~\ref{reversal}, we have that $$\varphi^{-m}(S)=\varphi^{-m}\left(S^{\rev}\right)=\left(\varphi^m(S)\right)^{\rev}=T^{\rev}=T.$$  Thus $S=\varphi^m(T)$, which implies \hbox{$\varphi^{2m}(S)=S$}.  Therefore, $\eso$ has $2m$ sets, since $m$ was chosen to be minimal.  Let $U\not=S,T$ be another set in $\eso$.  Then $U=\varphi^k(S)$ for some $k\in[2m-1]$ with $k\not=m$, and so $$\varphi^{-k}(S)=\varphi^{-k}\left(S^{\rev}\right)=\left(\varphi^k(S)\right)^{\rev}=U^{\rev}.$$  Since $\eso$ has $2m$ sets, we cannot have $\varphi^{-k}(S)=\varphi^k(S)$, so $U\not=U^{\rev}$.  Thus, by definition $U$ is not symmetrical, so $S$ and $T$ are the only symmetrical sets in $\eso$.
\end{proof}



\begin{lemma}\label{lem:21}Consider the action of $\varphi$.
\begin{enumerate}
\item When $S(i,j)=1$ and $j \not= n$, either $S(i,j+2)=1$ or $S(i+1,j+1)=1$, and never both.
\item When $S(i,j)=1$ and $j \not= 1$, either $S(i,j-2)=1$ or $S(i-1,j-1)=1$, and never both.
\item If $S(i,j)=1$, then $S(i,j-1)=S(i,j+1)=S(i-1,j)=S(i+1,j)=0$.
\item If $S(i,j)=1$, then $S(i+1,j-1)=S(i-1,j+1)=0$.
\end{enumerate}
\end{lemma}

\begin{proof}
(3) is clear because each $S^i$ is an independent set, and if $S(i,j)=1$, then $j\in S^i$, so $j\not\in \varphi(S^i)=S^{i+1}$.

Now we prove (1). If $j\in S^i$, then $j\not\in S^{i+1}$, so $j+1\in S^{i+1}$ if and only if $j+2\not\in S^i$.

The proof of (2) is analogous to that of (1) because $\varphi^{-1}$ applies the toggles in the reverse order.

To prove (4), assume $j\geq 2$ with $S(i,j)=1$.  Then $j\in S_i$.  Then after applying $\tau_{j-1}\cdots \tau_1$ to $S_i$, we cannot have both $j-1$ and $j$ in the same independent set.  So $S(i+1,j-1)=0$.  Thus if $S(i-1,j+1)=1$, then $S(i,j)=0$ which is a contradiction, so $S(i-1,j+1)=0$.
\end{proof}

From Lemma~\ref{lem:21}(1), given a 1 in the orbit board (outside of the rightmost column), there is another 1 either in the position two spaces to the right, or the position one space diagonally right and down.  From Lemma~\ref{lem:21}(2), for any 1 in the orbit board (outside of the leftmost column), there is another 1 either in the position two spaces to the left, or the position one space diagonally left and up.  Therefore, the ones in the orbit board can be partitioned into sequences, called \textbf{snakes}, that begin in the left column and end in the right column. For any 1 in the snake, the next 1 is located either two spaces to the right of it, or in the position one space diagonally right and down.

\begin{ex}\label{7snake}
The orbit board from Figure~\ref{fig:7ex}, with colors representing the different snakes, is shown in Figure~\ref{fig:7snakes}.
\begin{figure}
\begin{center}\begin{tabular}{c||c|c|c|c|c|c|c}
&\textbf{1}&\textbf{2}&\textbf{3}&\textbf{4}&\textbf{5}&\textbf{6}&\textbf{7}\\\hline\hline\\[-1.111em]
$S^0$&{\cellcolor{red}{\color{white}\textbf{1}}}&\textbf{0}&{\cellcolor{red}{\color{white}\textbf{1}}}&\textbf{0}&{\cellcolor{red}{\color{white}\textbf{1}}}&\textbf{0}&\textbf{0}\\\hline\\[-1.111em]
$S^1$&\textbf{0}&\textbf{0}&\textbf{0}&\textbf{0}&\textbf{0}&{\cellcolor{red}{\color{white}\textbf{1}}}&\textbf{0}\\\hline\\[-1.111em]
$S^2$&{\cellcolor{purple}{\color{white}\textbf{1}}}&\textbf{0}&{\cellcolor{purple}{\color{white}\textbf{1}}}&\textbf{0}&\textbf{0}&\textbf{0}&{\cellcolor{red}{\color{white}\textbf{1}}}\\\hline\\[-1.111em]
$S^3$&\textbf{0}&\textbf{0}&\textbf{0}&{\cellcolor{purple}{\color{white}\textbf{1}}}&\textbf{0}&\textbf{0}&\textbf{0}\\\hline\\[-1.111em]
$S^4$&{\cellcolor{green}{\color{white}\textbf{1}}}&\textbf{0}&\textbf{0}&\textbf{0}&{\cellcolor{purple}{\color{white}\textbf{1}}}&\textbf{0}&{\cellcolor{purple}{\color{white}\textbf{1}}}\\\hline\\[-1.111em]
$S^5$&\textbf{0}&{\cellcolor{green}{\color{white}\textbf{1}}}&\textbf{0}&\textbf{0}&\textbf{0}&\textbf{0}&\textbf{0}\\\hline\\[-1.111em]
$S^6$&\textbf{0}&\textbf{0}&{\cellcolor{green}{\color{white}\textbf{1}}}&\textbf{0}&{\cellcolor{green}{\color{white}\textbf{1}}}&\textbf{0}&{\cellcolor{green}{\color{white}\textbf{1}}}\\\hline\\[-1.111em]
$S^7$&{\cellcolor{blue}{\color{white}\textbf{1}}}&\textbf{0}&\textbf{0}&\textbf{0}&\textbf{0}&\textbf{0}&\textbf{0}\\\hline\\[-1.111em]
$S^8$&\textbf{0}&{\cellcolor{blue}{\color{white}\textbf{1}}}&\textbf{0}&{\cellcolor{blue}{\color{white}\textbf{1}}}&\textbf{0}&{\cellcolor{blue}{\color{white}\textbf{1}}}&\textbf{0}\\\hline\\[-1.111em]
$S^9$&\textbf{0}&\textbf{0}&\textbf{0}&\textbf{0}&\textbf{0}&\textbf{0}&{\cellcolor{blue}{\color{white}\textbf{1}}}\\\hline\hline\\[-1.111em]
\textbf{Total} & \textbf{4} & \textbf{2} & \textbf{3} & \textbf{2} & \textbf{3} & \textbf{2} & \textbf{4}
\end{tabular}\end{center}
\caption{The orbit from Figure~\ref{fig:7ex} with colors representing the snakes}
\label{fig:7snakes}
\end{figure}
\end{ex}

Therefore, to know where the ones in the orbit board are, it suffices to analyze the snakes. To each $\varphi$-orbit on $\cali_n$, we will associate an equivalence class of compositions of $n-1$ into parts 1 and 2, with each composition representing the snakes.

\begin{defn}
A \textbf{composition} of $n\in\zz^+$ is a sequence of positive integers whose sum is
$n$. Two compositions of $n$ are said to be \textbf{cyclically equivalent} if one is a
cyclic rotation of the other.  Otherwise, the compositions are \textbf{cyclically
inequivalent}.
\end{defn}

\begin{ex}
21121, 11212, 12121, 21211, and 12112 are cyclically equivalent compositions of 7.
\end{ex}

To associate a composition of $n-1$ to any given snake in a $\varphi$-orbit of $\cali_n$, a
step of two positions to the right corresponds to a 2, and a step of one position diagonally
right and down corresponds to a 1.  Thus, we get a composition of $n-1$ because we start in the
leftmost column and end in the rightmost column.

\begin{defn}
To associate a composition of $n-1$ to any given snake in a $\varphi$-orbit of $\cali_n$, a
step of two positions to the right corresponds to a 2, and a step of one position diagonally
right and down corresponds to a 1.  Thus, we get a composition of $n-1$ because we start in the
leftmost column and end in the rightmost column.  This is called the \textbf{snake composition} of the snake.
\end{defn}

In Example~\ref{7snake}, the red snake has snake composition 2211, the purple snake has snake composition 2112, the green snake has composition 1122, and the blue snake has composition 1221.

The following lemmata further constrain the possible pattern of ones in an orbit board.  They will be used in the proof of Theorem~\ref{thm:snakeorbit}.

\begin{lemma}\label{lem:1}
Under the action of $\varphi$, suppose $S(i,j)=1$ and $S(i+2,j-1)=1$.\begin{enumerate}
\item If $S(i,j+2)=1$, then $S(i+2,j+1)=1$.
\item If $S(i+1,j+1)=1$, then $S(i+3,j)=1$.
\item If $j=n$, then $S(i+3,j)=S(i+3,n)=1$.
\end{enumerate}
\end{lemma}


\begin{proof}
Without loss of generality we may assume $i=0$, as we can start our orbit board anywhere.  This means that $j\in S^0$ and $j-1\in S^2$.\begin{enumerate}

\item In this scenario, $j+2\in S^0$, and we wish to conclude that $j+1\in S^2$.  By Lemma~\ref{lem:21}(1), $S(0,j)=S(0,j+2)=1$ gives $S(1,j+1)=0$.  Thus $j+1\not\in S^1$.  Also, $j+2\in S^0$ implies $j+2\not\in S^1$. And $j-1\in S^2$ implies $j\not\in S^2$.  Therefore, when applying toggles to $S^1$, $j+1$ gets toggled in, so $j+1\in S^2$.

\item In this scenario, $j+1\in S^1$, and we wish to conclude that $j\in S^3$.  Since $j-1\in S^2$, we can use Lemma~\ref{lem:21}(1) to determine that either $j+1\in S^2$ or $j\in S^3$.  However, $j+1\not\in S^2$ because $j+1\in S^1$.  Therefore, $j\in S^3$.

\item Since $S(2,n-1)=1$, either $S(2,n+1)=1$ or $S(3,n)=1$ from Lemma~\ref{lem:21}(1).  Only the second scenario is possible so $S(3,n)=1$.
\end{enumerate}
\end{proof}

\begin{lemma}\label{lem:2}
Under the action of $\varphi$, suppose $S(i,j)=1$ and $S(i+2,j-2)=1$.\begin{enumerate}
\item If $S(i,j+2)=1$, then $S(i+2,j)=1$.
\item If $S(i+1,j+1)=1$, then $S(i+3,j-1)=1$.
\item If $j=n$, then $S(i+2,j)=S(i+2,n)=1$.
\end{enumerate}
\end{lemma}


\begin{proof}
As in the proof of the previous lemma, assume $i=0$ without loss of generality.  This means that $j\in S^0$ and $j-2\in S^2$.
\begin{enumerate}
\item In this scenario, $j+2\in S^0$, and we wish to conclude that $j\in S^2$.  Since
$S(0,j) =S(0,j+2)=1$, we conclude from Lemma~\ref{lem:21}(1) that $S(1,j+1)=0$ and so $j+1\not\in S^1$.  Note that $j-2\in S^2$ gives $j-1\not\in S^2$.  Also $j\in S^0$ gives $j\not\in S^1$.  Thus, when we apply toggles left to right starting with $S^1$, we will be able to add vertex $j$ to the set.  Thus, $j\in S^2$.

\item In this scenario, $j+1\in S^1$, and we wish to conclude that $j-1\in S^3$.  Since $j+1\in S^1$, it follows that $j\not\in S^2$ by Lemma~\ref{lem:21}(4).  Since $S(2,j-2)=1$ and $S(2,j)=0$, we have $S(3,j-1)=1$ by Lemma~\ref{lem:21}(2).

\item Since $n\in S$, we have $n\not\in S^1$.  Also, $n-2\in S^2$ implies $n-1\not\in S^2$.  Thus, when we reach the last vertex when applying $\varphi$ to $S^1$, we insert $n$.  So $n\in S^2$.
\end{enumerate}
\end{proof}

\begin{thm}\label{thm:snakeorbit}
In a $\varphi$-orbit board, consider a snake starting on the $S^i$ line. Let $c$ be the snake's composition. Consider the least $i'>i$ for which $S(i',1)=1$. (This is where the ``next" snake begins.)
\begin{enumerate}
\item If $c$ starts with 1, then $i'=i+3$.
\item If $c$ starts with 2, then $i'=i+2$.
\item The composition for the snake starting on the $S^{i'}$ line is the left cyclic rotation of $c$.
\end{enumerate}

\end{thm}

\begin{proof}
Without loss of generality, assume $i=0$. Let $c'$ be the composition for the snake that starts on line $S^{i'}$.

If $c$ starts with 1, then $S(0,1)=S(1,2)=1$. Then $1\not\in\tau_1(S^1)$ because $2\in S^1$. Then applying $\tau_2$ to $S_1$ removes 2 from the set.  Therefore $1,2\not\in S^2$, and thus we insert 1 when applying $\tau_1$ to $S^2$.  So $S(3,1)=1$, which proves (1). The part of $c$ after the initial 1 describes the sequence of moves for the original snake from $S(1,2)$ to the rightmost column. Since $S(1,2)=S(3,1)=1$, this same sequence of moves describes the snake that starts on line $S^3$ from the leftmost column up to column $n-1$, by Lemma~\ref{lem:1}(1,2). Then the snake with composition $c'$ must finish with a diagonal step, so $c'$ ends with 1.  Thus $c'$ is formed from $c$ by moving the initial 1 to the end.  This proves (3) for the case where $c$ begins with 1.

Otherwise $c$ starts with 2, so $S(0,1)=S(0,3)=1$. By Lemma~\ref{lem:21}(3), $S(1,1)=0$ and by Lemma~\ref{lem:21}(1), $S(1,2)=0$. Since $1,2\not\in S^1$, we insert 1 when applying $\tau_1$ to $S^1$, so $S(2,1)=1$, which proves (2). The part of $c$ after the initial 2 describes the sequence of moves for the original snake from $S(0,3)$ to the rightmost column. Since $S(0,3)=S(2,1)=1$, this same sequence of moves describes the snake that starts on line $S^2$ from the leftmost column up to column $n-2$, by Lemma~\ref{lem:2}(1,2). Suppose the snake with composition $c$ ends on line $S^k$, i.e. $S(k,n)=1$.  Then the snake starting on line $S^2$ contains $S(k+2,n-2)$, so by Lemma~\ref{lem:2}(3), $S(k+2,n)=1$.  Thus $c'$ is formed from $c$ by moving the initial 2 to the end.  This proves (3) for the case where $c$ begins with 2.

\end{proof}

\begin{ex}\label{board}
We show how knowing one snake determines an entire orbit.  Suppose we are working in
$\cali_{10}$ and we have a snake given by the composition 221121.  Then we immediately have
the part of the orbit board shown in Figure~\ref{fig:snake}.
\begin{figure}
\begin{center}\begin{tabular}{c||c|c|c|c|c|c|c|c|c|c}
&\textbf{1}&\textbf{2}&\textbf{3}&\textbf{4}&\textbf{5}&\textbf{6}&\textbf{7}&\textbf{8}&\textbf{9}&\textbf{10}\\\hline\hline\\[-1.111em]
$S^0$&{\cellcolor{red}{\color{white}\textbf{1}}}&&{\cellcolor{red}{\color{white}\textbf{1}}}&&{\cellcolor{red}{\color{white}\textbf{1}}}&&&&&\\\hline\\[-1.111em]
$S^1$&&&&&&{\cellcolor{red}{\color{white}\textbf{1}}}&&&&\\\hline\\[-1.111em]
$S^2$&&&&&&&{\cellcolor{red}{\color{white}\textbf{1}}}&&{\cellcolor{red}{\color{white}\textbf{1}}}&\\\hline\\[-1.111em]
$S^3$&&&&&&&&&&{\cellcolor{red}{\color{white}\textbf{1}}}\\
\end{tabular}\end{center}
\caption{The ones in the orbit board are an example snake.  In Example~\ref{board}, we describe how to generate an entire orbit from one snake.}
\label{fig:snake}
\end{figure}

Using Theorem~\ref{thm:snakeorbit}, we know that the next snake begins on the $S^2$ line, and has snake composition 211212.  This snake is shown in purple in Figure~\ref{fig:10orbit}.
Also by Theorem~\ref{thm:snakeorbit}, the next four snakes start on the lines have snake compositions 112122, 121221, 212211, and 122112 respectively and begin on lines $S^4$, $S^7$, $S^{10}$ and $S^{12}$.  These are shown in orange, green, blue, and brown respectively in Figure~\ref{fig:10orbit}.

Then the next snake starts on the $S^{15}$ line and has snake composition 221121.  However,
this is the snake we started with.  Therefore $S^0=S^{15}$. So this orbit has size 15, and the
ones in the brown snake on the $S^{15}$ line go on the $S^0$ line.  Every other empty position
is a 0 by Lemma~\ref{lem:21}(3).  The full orbit board is shown in Figure~\ref{fig:10orbit}.
Notice that this orbit is not reversible, so there is no simple reason for the column-sum
vector to be palindromic.
\begin{figure}
\begin{center}
\begin{minipage}{0.58\textwidth}
\begin{tabular}{c||c|c|c|c|c|c|c|c|c|c}
&\textbf{1}&\textbf{2}&\textbf{3}&\textbf{4}&\textbf{5}&\textbf{6}&\textbf{7}&\textbf{8}&\textbf{9}&\textbf{10}\\\hline\hline\\[-1.111em]
$S^0$&{\cellcolor{red}{\color{white}\textbf{1}}}&\textbf{0}&{\cellcolor{red}{\color{white}\textbf{1}}}&\textbf{0}&{\cellcolor{red}{\color{white}\textbf{1}}}&\textbf{0}&\textbf{0}&{\cellcolor{brown}{\color{white}\textbf{1}}}&\textbf{0}&{\cellcolor{brown}{\color{white}\textbf{1}}}\\\hline\\[-1.111em]
$S^1$&\textbf{0}&\textbf{0}&\textbf{0}&\textbf{0}&\textbf{0}&{\cellcolor{red}{\color{white}\textbf{1}}}&\textbf{0}&\textbf{0}&\textbf{0}&\textbf{0}\\\hline\\[-1.111em]
$S^2$&{\cellcolor{purple}{\color{white}\textbf{1}}}&\textbf{0}&{\cellcolor{purple}{\color{white}\textbf{1}}}&\textbf{0}&\textbf{0}&\textbf{0}&{\cellcolor{red}{\color{white}\textbf{1}}}&\textbf{0}&{\cellcolor{red}{\color{white}\textbf{1}}}&\textbf{0}\\\hline\\[-1.111em]
$S^3$&\textbf{0}&\textbf{0}&\textbf{0}&{\cellcolor{purple}{\color{white}\textbf{1}}}&\textbf{0}&\textbf{0}&\textbf{0}&\textbf{0}&\textbf{0}&{\cellcolor{red}{\color{white}\textbf{1}}}\\\hline\\[-1.111em]
$S^4$&{\cellcolor{orange}{\color{white}\textbf{1}}}&\textbf{0}&\textbf{0}&\textbf{0}&{\cellcolor{purple}{\color{white}\textbf{1}}}&\textbf{0}&{\cellcolor{purple}{\color{white}\textbf{1}}}&\textbf{0}&\textbf{0}&\textbf{0}\\\hline\\[-1.111em]
$S^5$&\textbf{0}&{\cellcolor{orange}{\color{white}\textbf{1}}}&\textbf{0}&\textbf{0}&\textbf{0}&\textbf{0}&\textbf{0}&{\cellcolor{purple}{\color{white}\textbf{1}}}&\textbf{0}&{\cellcolor{purple}{\color{white}\textbf{1}}}\\\hline\\[-1.111em]
$S^6$&\textbf{0}&\textbf{0}&{\cellcolor{orange}{\color{white}\textbf{1}}}&\textbf{0}&{\cellcolor{orange}{\color{white}\textbf{1}}}&\textbf{0}&\textbf{0}&\textbf{0}&\textbf{0}&\textbf{0}\\\hline\\[-1.111em]
$S^7$&{\cellcolor{green}{\color{white}\textbf{1}}}&\textbf{0}&\textbf{0}&\textbf{0}&\textbf{0}&{\cellcolor{orange}{\color{white}\textbf{1}}}&\textbf{0}&{\cellcolor{orange}{\color{white}\textbf{1}}}&\textbf{0}&{\cellcolor{orange}{\color{white}\textbf{1}}}\\\hline\\[-1.111em]
$S^8$&\textbf{0}&{\cellcolor{green}{\color{white}\textbf{1}}}&\textbf{0}&{\cellcolor{green}{\color{white}\textbf{1}}}&\textbf{0}&\textbf{0}&\textbf{0}&\textbf{0}&\textbf{0}&\textbf{0}\\\hline\\[-1.111em]
$S^9$&\textbf{0}&\textbf{0}&\textbf{0}&\textbf{0}&{\cellcolor{green}{\color{white}\textbf{1}}}&\textbf{0}&{\cellcolor{green}{\color{white}\textbf{1}}}&\textbf{0}&{\cellcolor{green}{\color{white}\textbf{1}}}&\textbf{0}\\\hline\\[-1.111em]
$S^{10}$&{\cellcolor{blue}{\color{white}\textbf{1}}}&\textbf{0}&{\cellcolor{blue}{\color{white}\textbf{1}}}&\textbf{0}&\textbf{0}&\textbf{0}&\textbf{0}&\textbf{0}&\textbf{0}&{\cellcolor{green}{\color{white}\textbf{1}}}\\\hline\\[-1.111em]
$S^{11}$&\textbf{0}&\textbf{0}&\textbf{0}&{\cellcolor{blue}{\color{white}\textbf{1}}}&\textbf{0}&{\cellcolor{blue}{\color{white}\textbf{1}}}&\textbf{0}&{\cellcolor{blue}{\color{white}\textbf{1}}}&\textbf{0}&\textbf{0}\\\hline\\[-1.111em]
$S^{12}$&{\cellcolor{brown}{\color{white}\textbf{1}}}&\textbf{0}&\textbf{0}&\textbf{0}&\textbf{0}&\textbf{0}&\textbf{0}&\textbf{0}&{\cellcolor{blue}{\color{white}\textbf{1}}}&\textbf{0}\\\hline\\[-1.111em]
$S^{13}$&\textbf{0}&{\cellcolor{brown}{\color{white}\textbf{1}}}&\textbf{0}&{\cellcolor{brown}{\color{white}\textbf{1}}}&\textbf{0}&{\cellcolor{brown}{\color{white}\textbf{1}}}&\textbf{0}&\textbf{0}&\textbf{0}&{\cellcolor{blue}{\color{white}\textbf{1}}}\\\hline\\[-1.111em]
$S^{14}$&\textbf{0}&\textbf{0}&\textbf{0}&\textbf{0}&\textbf{0}&\textbf{0}&{\cellcolor{brown}{\color{white}\textbf{1}}}&\textbf{0}&\textbf{0}&\textbf{0}\\\hline\hline\\[-1.111em]
\textbf{Total} & \textbf{6} & \textbf{3} & \textbf{4} & \textbf{4} & \textbf{4} & \textbf{4} & \textbf{4} & \textbf{4} & \textbf{3} & \textbf{6}
\end{tabular}\\\vspace{0.1 in}
\end{minipage}
\begin{minipage}{0.38\textwidth}
\begin{tabular}{lc}
\textbf{{\color{red}Red snake:}} &\textbf{{\color{red}221121}}\\
&\\\\[-1.111em]\\[-1.111em]
\textbf{{\color{purple}Purple snake:}} &\textbf{{\color{purple}211212}}\\
&\\ \\[-1.111em]\\[-1.111em]
\textbf{{\color{orange}Orange snake:}} &\textbf{{\color{orange}112122}}\\
&\\\\[-1.111em]\\[-1.111em]\\[-1.111em]
&\\[4pt]
\textbf{{\color{green}Green snake:}} &\textbf{{\color{green}121221}}\\
&\\\\[-1.111em]\\[-1.111em]\\[-1.111em]
&\\
\textbf{{\color{blue}Blue snake:}} &\textbf{{\color{blue}212211}}\\
&\\ \\[-1.111em]\\[-1.111em]
\textbf{{\color{brown}Brown snake:}} &\textbf{{\color{brown}122112}}\\
&\\
&\\
&\\
\end{tabular}
\end{minipage}
\end{center}
\caption{The unique orbit containing the snake from Figure~\ref{fig:snake} is the orbit
containing $S=1010100101$ (See Example~\ref{board}).}
\label{fig:10orbit}
\end{figure}
\end{ex}

The following should be clear now.

\begin{prop}\label{prop:snakebijection}
For any $\varphi$-orbit, the set of snake compositions is invariant under cyclic rotation.
Thus,
there is a bijection between $\varphi$-orbits of $\cali_n$ and cyclically inequivalent
compositions of $n-1$ into parts 1 and 2.  Also, an orbit $\eso$ is reversible if and only
if for each snake in the orbit with snake composition $c$, there is also a snake in the orbit
whose snake composition is $c$ in the reverse order.
\end{prop}

We are now ready to prove Propp's original conjecture.

\begin{proof}[Proof of Theorem \ref{indepsethomomesy}]
We wish to prove that for any $j\in[n]$, $\chi_j-\chi_{n+1-j}$ is 0-mesic.  It suffices to
show that for any
$\eso$, $$\sum\limits_{S\in\eso}\chi_j(S)=\sum\limits_{S\in\eso}\chi_{n+1-j}(S).$$

Since every snake in $\eso$ starts in the leftmost column and ends in the rightmost column,
the orbit has the same number of ones in the leftmost column as in the rightmost column of
the (finite version of the) orbit board
Thus, $$\sum\limits_{S\in\eso}\chi_1(S)=\sum\limits_{S\in\eso}\chi_n(S).$$
Each entry in the snake composition refers to how many columns we move right to get to the next 1 in the snake.
Thus for $j>1$,  there is a 1 in column $j$ of the orbit board for every snake
composition that has an initial segment adding to $j-1$. Similarly, there is a 1 in column
$n+1-j$ of the orbit board for every snake composition that has a final segment adding to
$j-1$.  By cyclic rotation of snake compositions, we get that there are the same number of
snake compositions in $\eso$ with an initial segment that adds to $j-1$ as there are with a
final segment that adds to $j-1$.
\end{proof}

\begin{ex}
For the orbit board in Example~\ref{board}, there is a 1 in column 4 whenever an initial segment
of a snake's composition adds to 3.
There are two snake compositions associated with this orbit that begin with 12.  They are
121221 ({\color{green} green}) and 122112 ({\color{brown} brown}).  For each of these, there
is a cyclic rotation of the snake composition that ends with 12.  These are 122112
({\color{brown} brown}) and 211212 ({\color{purple} purple}).  These give ones in column 7
(fourth column from the right).

Also there are two snake compositions associated with this orbit that begin with 21.  They
are 211212 ({\color{purple} purple}) and 212211 ({\color{blue} blue}).  For each of these,
there is a cyclic rotation of the snake composition that ends with 21.  These are 121221
({\color{green} green}) and 221121 ({\color{red} red}).  These give ones in column 7 (fourth
column from the right).

If there were snake compositions for this orbit that began with 111 (the other way to have
an initial segment adding to 3), then by cyclic rotation, there would be just as many that
end in 111.
\end{ex}

\subsection{Coxeter groups and extending Propp's conjecture}\label{subsec:Coxeter}

Theorems~\ref{indepsethomomesy} and~\ref{babyhomomesy} are for orbits of the specific action
$\varphi=\tau_n\cdots\tau_2\tau_1$. In this subsection, we describe how these homomesy results
can be generalized to some other actions in $\calt_n$, namely \emph{Coxeter elements}.

\begin{defn}
An element $w\in\calt_n$ is called a \textbf{Coxeter element} if it is a product of $\tau_1,\tau_2,\dots,\tau_n$, each used exactly once, in some order.
\end{defn}

\begin{thm}\label{thm:preserve-Coxeter}
Let $w,w'\in\calt_n$ be two Coxeter elements.  Any statistic which is a linear combination of the indicator functions $\chi_j$ is $c$-mesic under the action of $w$ if and only if it is $c$-mesic under the action of $w'$.
\end{thm}

Since $\varphi$ is a Coxeter element in $\calt_n$, we use Theorem~\ref{thm:preserve-Coxeter} to generalize
Theorems~\ref{indepsethomomesy} and~\ref{babyhomomesy} to the action of \emph{any} Coxeter element.  This is the following corollary.

\begin{cor}\label{cor:hom-Coxeter} Let $w\in\calt_n$ be a Coxeter element, and consider the
action of $w$ on $\cali_n$.
\begin{enumerate}
\item The statistic $\chi_j-\chi_{n+1-j}$ is 0-mesic for every $1\leq j\leq n$.
\item The statistics $2\chi_1+\chi_2$ and $\chi_{n-1}+2\chi_n$ are both 1-mesic.
\end{enumerate}
\end{cor}

The rest of this section leads up to the proof of Theorem~\ref{thm:preserve-Coxeter}, which is near the end.

As $\calt_n$ is generated by finitely many involutions, it is the quotient of a Coxeter group.  See~\cite{bjornerbrenti} and~\cite[Ch.~11--14]{peterseneulerian} to learn about
Coxeter groups and their connections to combinatorics.  The Coxeter group with
presentation $$\left<\tau_1,\tau_2,\dots,\tau_n|\tau_i^2=1, (\tau_i\tau_{i+1})^6=1,
(\tau_i\tau_j)^2=1 \text{ for }j\geq i+2\right>,$$ for $n\geq 3$, is well known to be
infinite.  However, the toggle group $\calt_n$ is a subgroup of $\ss_{\cali_n}$ and thus
finite. Therefore $\calt_n$ has extra relations in addition to the Coxeter relations
($\tau_i^2=1$, $(\tau_i\tau_i+1)^6=1$, $(\tau_i\tau_j)^2=1$ for $j\geq i+2$), so it is not a
true Coxeter group.  Knowledge of Coxeter groups is not necessary to understand the rest of
the paper, as we explain the theory necessary to describe why we can extend the homomesy results for $\varphi$ to orbits of any Coxeter element $w\in\calt_n$.  Even though $\calt_n$ is not a true Coxeter group, we borrow the term ``Coxeter element'' from Coxeter group theory.

The following is a brief summary of material in~\cite[\S3,6]{efgjmpr}, which describes in detail how Coxeter group theory can be applied to toggle groups, and the connections between conjugation of elements and homomesy.  We discuss it in the context of $\calt_n$.

To each Coxeter element $w\in\calt_n$, we define an orientation of $\calp_n$. That is, we direct each edge of the path graph in such a way that corresponds to our Coxeter element.  To do this, we direct the edge connecting $i$ and $i+1$ in the direction of $i$ if $\tau_i$ appears to the right of $\tau_{i+1}$ in an expression of $w$, and in the direction of $i+1$ if $\tau_{i+1}$ appears to the right of $\tau_i$ in an expression of $w$.

\begin{ex}\label{ex:3426751}
The Coxeter element $w=\tau_3\tau_4\tau_2\tau_6\tau_7\tau_5\tau_1 \in \calt_7$ corresponds to the following orientation of $\calp_7$.

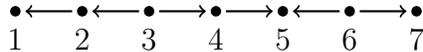
\begin{center}\begin{tikzpicture}[scale=8/9]
\draw[thick, <-] (1.15,0) -- (1.85,0);
\draw[thick, <-] (2.15,0) -- (2.85,0);
\draw[thick, ->] (3.15,0) -- (3.85,0);
\draw[thick, ->] (4.15,0) -- (4.85,0);
\draw[thick, <-] (5.15,0) -- (5.85,0);
\draw[thick, ->] (6.15,0) -- (6.85,0);
\draw[fill] (1,0) circle [radius=0.07];
\draw[fill] (2,0) circle [radius=0.07];
\draw[fill] (3,0) circle [radius=0.07];
\draw[fill] (4,0) circle [radius=0.07];
\draw[fill] (5,0) circle [radius=0.07];
\draw[fill] (6,0) circle [radius=0.07];
\draw[fill] (7,0) circle [radius=0.07];
\node[below] at (1,-0.1) {1};
\node[below] at (2,-0.1) {2};
\node[below] at (3,-0.1) {3};
\node[below] at (4,-0.1) {4};
\node[below] at (5,-0.1) {5};
\node[below] at (6,-0.1) {6};
\node[below] at (7,-0.1) {7};
\end{tikzpicture}
\end{center}

While $w$ has other representations as a product of toggles, each used exactly once, they are all found by taking the expression $\tau_3\tau_4\tau_2\tau_6\tau_7\tau_5\tau_1$ and applying the commutativity relations $\tau_i\tau_j=\tau_j\tau_i$ if $|i-j|\geq2$.  For example, $w=\tau_3\tau_2\tau_6\tau_4\tau_5\tau_7\tau_1$, found by moving $\tau_4$ to the right of $\tau_2\tau_6$ and swapping the order of $\tau_5$ and $\tau_7$.  Since applying the commutativity relations does not change the relative order of any pair $\tau_i,\tau_{i+1}$, this does not change the orientation of any edge.  Thus, the orientation of $\calp_n$ corresponds uniquely to the Coxeter element, and this orientation is well-defined.
\end{ex}

Shi showed that Coxeter elements correspond uniquely to acyclic orientations of the graph associated with a Coxeter group, where the vertices of the graph are the involution generators, and the edges of the graph are between pairs of noncommuting generators~\cite[Proposition~1.3]{shi1997enumeration}.  Since $\calp_n$ already has no cycles, any orientation of it corresponds to a unique Coxeter element in $\calt_n$, even though $\calt_n$ is only the quotient of a Coxeter group.  Toggle maps corresponding to source vertices are called \textbf{initial} in $w$ because they can be brought to the left by the commutativity relations.  Toggle maps corresponding to sink vertices are called \textbf{final} in $w$ because they can be brought to the right by the commutativity relations.  Vertices that are neither sources nor sinks correspond to toggles that can neither be brought to the left nor right in an expression of $w$ (while using every toggle exactly once).

Since the toggles are involutions, conjugating by an final (resp. initial) toggle $\tau_i$ of $w$ corresponds to moving it to the left (resp. right) of an expression for $w$.  For example, $\tau_5$ is final in $w=\tau_3\tau_2\tau_6\tau_4\tau_5\tau_7\tau_1$, so conjugating by $\tau_5$ gives \begin{eqnarray*}
\tau_5 w \tau_5 &=& \tau_5(\tau_3\tau_2\tau_6\tau_4\tau_5\tau_7\tau_1)\tau_5\\
&=& \tau_5(\tau_3\tau_2\tau_6\tau_4\tau_7\tau_1\tau_5)\tau_5\\
&=&\tau_5\tau_3\tau_2\tau_6\tau_4\tau_7\tau_1.
\end{eqnarray*}

This conjugation by a final element of $w$ corresponds to changing a sink into a source in the corresponding orientation of $\calp_n$.  Conjugation by an initial element would correspond to changing a source into a sink.  If we were to conjugate $w$ by $\tau_4$, which is neither initial nor final, we would get $\tau_4\tau_3\tau_2\tau_6\tau_4\tau_5\tau_7\tau_1\tau_4$ which is not a Coxeter element.

H. and K. Eriksson showed that two Coxeter elements are conjugate if and only if we can
transform the orientation for one of them into the other by a sequence that changes sinks
into sources or vice versa~\cite{eriksson2009conjugacy}.  This corresponds to a sequence of
conjugations by generators that are either initial or final; we call such conjugations
\textbf{admissible}.

In the case of $\calt_n$, any two Coxeter elements are conjugate.  This is because $\calt_n$ has generators $\tau_1,\tau_2,\dots,\tau_n$ satisfying $\tau_i^2=1$ and $(\tau_i\tau_j)^2=1$ when $|i-j|>1$.  Thus, any Coxeter elements in $\calt_n$ are conjugate~\cite[Lemma 5.1]{strikerwilliams}.  Not only that, but the proof given by Striker and Williams shows that we can get from any Coxeter element to any other one by a sequence of \emph{admissible} conjugations.  We describe their method for conjugating by toggles to transform any Coxeter element $w$ into $\varphi=\tau_n\cdots\tau_2\tau_1$.  The method is slightly modified here because we wish to transform $w$ into $\tau_n\cdots\tau_2\tau_1$ not $\tau_1\tau_2\cdots\tau_n$.  Starting with $w$, find the largest number $k$ such that $\tau_k$ is final in $w$, then push $\tau_k$ to the right and conjugate by $\tau_k$.  Repeat this until we arrive at $\varphi=\tau_n\cdots\tau_2\tau_1$.

\begin{ex}\label{ex:conjugation}
Let $w=\tau_3\tau_4\tau_2\tau_6\tau_7\tau_5\tau_1$ as in Example~\ref{ex:3426751}.  Refer to Figure~\ref{fig:3426751} for a sequence of conjugations to transform $w$ into $\varphi$ and the corresponding orientations of $\calp_n$ at each step.  By the process described above, $\varphi=u^{-1}wu$ where $u=\tau_7\tau_5\tau_6\tau_7\tau_4\tau_5\tau_6\tau_7$.
\end{ex}

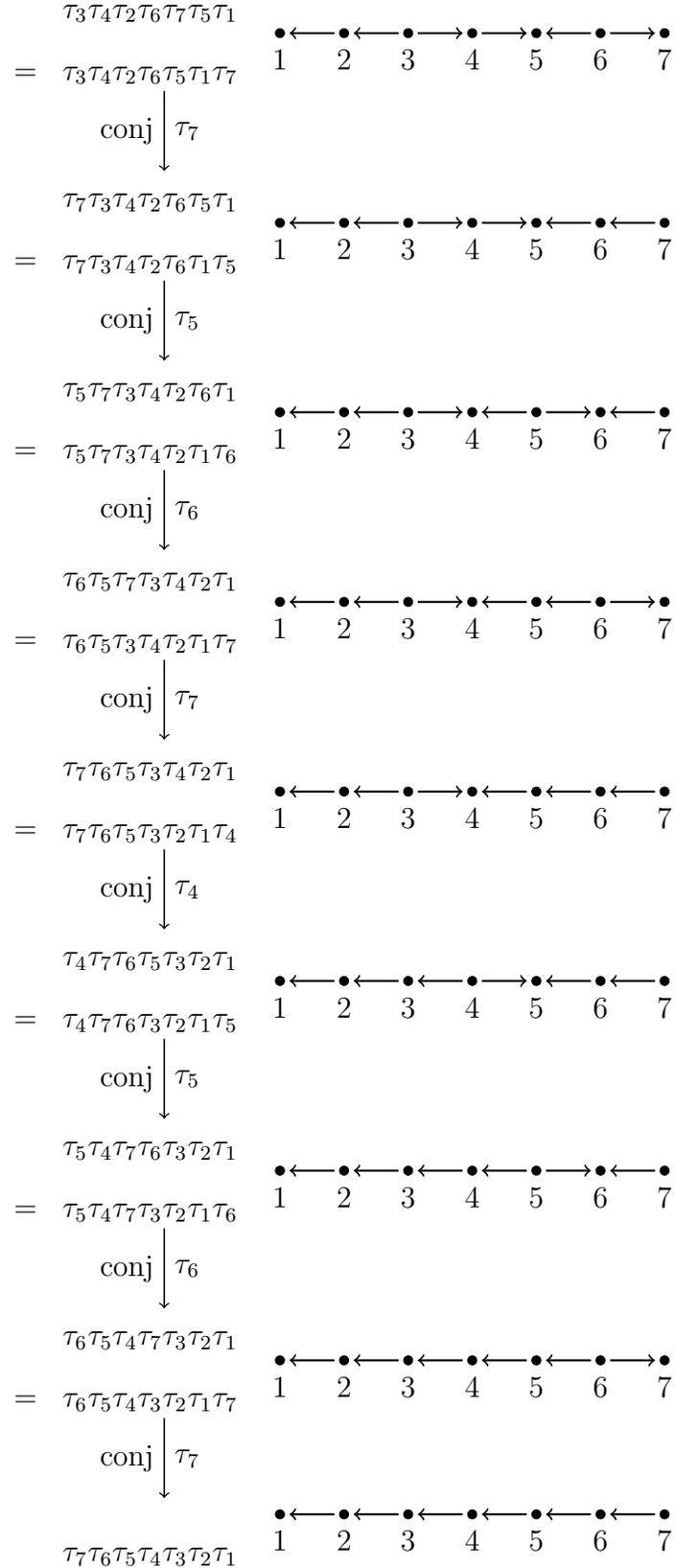
\begin{figure}
\begin{center}
\begin{tabular}{ccc}
&$\tau_3\tau_4\tau_2\tau_6\tau_7\tau_5\tau_1$&\\
$=$&$\tau_3\tau_4\tau_2\tau_6\tau_5\tau_1\tau_7$&\begin{tikzpicture}[scale=8/9]
\draw[thick, <-] (1.15,0) -- (1.85,0);
\draw[thick, <-] (2.15,0) -- (2.85,0);
\draw[thick, ->] (3.15,0) -- (3.85,0);
\draw[thick, ->] (4.15,0) -- (4.85,0);
\draw[thick, <-] (5.15,0) -- (5.85,0);
\draw[thick, ->] (6.15,0) -- (6.85,0);
\draw[fill] (1,0) circle [radius=0.07];
\draw[fill] (2,0) circle [radius=0.07];
\draw[fill] (3,0) circle [radius=0.07];
\draw[fill] (4,0) circle [radius=0.07];
\draw[fill] (5,0) circle [radius=0.07];
\draw[fill] (6,0) circle [radius=0.07];
\draw[fill] (7,0) circle [radius=0.07];
\node[below] at (1,-0.1) {1};
\node[below] at (2,-0.1) {2};
\node[below] at (3,-0.1) {3};
\node[below] at (4,-0.1) {4};
\node[below] at (5,-0.1) {5};
\node[below] at (6,-0.1) {6};
\node[below] at (7,-0.1) {7};
\end{tikzpicture}\\

&\begin{tikzpicture}
\draw[semithick, ->] (0,10/9) -- (0,0);
\node[left] at (0,5/9) {conj};
\node[right] at (0,5/9) {$\tau_7$};
\end{tikzpicture}&\\

&$\tau_7\tau_3\tau_4\tau_2\tau_6\tau_5\tau_1$&\\
$=$&$\tau_7\tau_3\tau_4\tau_2\tau_6\tau_1\tau_5$&\begin{tikzpicture}[scale=8/9]
\draw[thick, <-] (1.15,0) -- (1.85,0);
\draw[thick, <-] (2.15,0) -- (2.85,0);
\draw[thick, ->] (3.15,0) -- (3.85,0);
\draw[thick, ->] (4.15,0) -- (4.85,0);
\draw[thick, <-] (5.15,0) -- (5.85,0);
\draw[thick, <-] (6.15,0) -- (6.85,0);
\draw[fill] (1,0) circle [radius=0.07];
\draw[fill] (2,0) circle [radius=0.07];
\draw[fill] (3,0) circle [radius=0.07];
\draw[fill] (4,0) circle [radius=0.07];
\draw[fill] (5,0) circle [radius=0.07];
\draw[fill] (6,0) circle [radius=0.07];
\draw[fill] (7,0) circle [radius=0.07];
\node[below] at (1,-0.1) {1};
\node[below] at (2,-0.1) {2};
\node[below] at (3,-0.1) {3};
\node[below] at (4,-0.1) {4};
\node[below] at (5,-0.1) {5};
\node[below] at (6,-0.1) {6};
\node[below] at (7,-0.1) {7};
\end{tikzpicture}\\

&\begin{tikzpicture}
\draw[semithick, ->] (0,10/9) -- (0,0);
\node[left] at (0,5/9) {conj};
\node[right] at (0,5/9) {$\tau_5$};
\end{tikzpicture}&\\

&$\tau_5\tau_7\tau_3\tau_4\tau_2\tau_6\tau_1$&\\
$=$&$\tau_5\tau_7\tau_3\tau_4\tau_2\tau_1\tau_6$&\begin{tikzpicture}[scale=8/9]
\draw[thick, <-] (1.15,0) -- (1.85,0);
\draw[thick, <-] (2.15,0) -- (2.85,0);
\draw[thick, ->] (3.15,0) -- (3.85,0);
\draw[thick, <-] (4.15,0) -- (4.85,0);
\draw[thick, ->] (5.15,0) -- (5.85,0);
\draw[thick, <-] (6.15,0) -- (6.85,0);
\draw[fill] (1,0) circle [radius=0.07];
\draw[fill] (2,0) circle [radius=0.07];
\draw[fill] (3,0) circle [radius=0.07];
\draw[fill] (4,0) circle [radius=0.07];
\draw[fill] (5,0) circle [radius=0.07];
\draw[fill] (6,0) circle [radius=0.07];
\draw[fill] (7,0) circle [radius=0.07];
\node[below] at (1,-0.1) {1};
\node[below] at (2,-0.1) {2};
\node[below] at (3,-0.1) {3};
\node[below] at (4,-0.1) {4};
\node[below] at (5,-0.1) {5};
\node[below] at (6,-0.1) {6};
\node[below] at (7,-0.1) {7};
\end{tikzpicture}\\

&\begin{tikzpicture}
\draw[semithick, ->] (0,10/9) -- (0,0);
\node[left] at (0,5/9) {conj};
\node[right] at (0,5/9) {$\tau_6$};
\end{tikzpicture}&\\

&$\tau_6\tau_5\tau_7\tau_3\tau_4\tau_2\tau_1$&\\
$=$&$\tau_6\tau_5\tau_3\tau_4\tau_2\tau_1\tau_7$&\begin{tikzpicture}[scale=8/9]
\draw[thick, <-] (1.15,0) -- (1.85,0);
\draw[thick, <-] (2.15,0) -- (2.85,0);
\draw[thick, ->] (3.15,0) -- (3.85,0);
\draw[thick, <-] (4.15,0) -- (4.85,0);
\draw[thick, <-] (5.15,0) -- (5.85,0);
\draw[thick, ->] (6.15,0) -- (6.85,0);
\draw[fill] (1,0) circle [radius=0.07];
\draw[fill] (2,0) circle [radius=0.07];
\draw[fill] (3,0) circle [radius=0.07];
\draw[fill] (4,0) circle [radius=0.07];
\draw[fill] (5,0) circle [radius=0.07];
\draw[fill] (6,0) circle [radius=0.07];
\draw[fill] (7,0) circle [radius=0.07];
\node[below] at (1,-0.1) {1};
\node[below] at (2,-0.1) {2};
\node[below] at (3,-0.1) {3};
\node[below] at (4,-0.1) {4};
\node[below] at (5,-0.1) {5};
\node[below] at (6,-0.1) {6};
\node[below] at (7,-0.1) {7};
\end{tikzpicture}\\

&\begin{tikzpicture}
\draw[semithick, ->] (0,10/9) -- (0,0);
\node[left] at (0,5/9) {conj};
\node[right] at (0,5/9) {$\tau_7$};
\end{tikzpicture}&\\

&$\tau_7\tau_6\tau_5\tau_3\tau_4\tau_2\tau_1$&\\
$=$&$\tau_7\tau_6\tau_5\tau_3\tau_2\tau_1\tau_4$&\begin{tikzpicture}[scale=8/9]
\draw[thick, <-] (1.15,0) -- (1.85,0);
\draw[thick, <-] (2.15,0) -- (2.85,0);
\draw[thick, ->] (3.15,0) -- (3.85,0);
\draw[thick, <-] (4.15,0) -- (4.85,0);
\draw[thick, <-] (5.15,0) -- (5.85,0);
\draw[thick, <-] (6.15,0) -- (6.85,0);
\draw[fill] (1,0) circle [radius=0.07];
\draw[fill] (2,0) circle [radius=0.07];
\draw[fill] (3,0) circle [radius=0.07];
\draw[fill] (4,0) circle [radius=0.07];
\draw[fill] (5,0) circle [radius=0.07];
\draw[fill] (6,0) circle [radius=0.07];
\draw[fill] (7,0) circle [radius=0.07];
\node[below] at (1,-0.1) {1};
\node[below] at (2,-0.1) {2};
\node[below] at (3,-0.1) {3};
\node[below] at (4,-0.1) {4};
\node[below] at (5,-0.1) {5};
\node[below] at (6,-0.1) {6};
\node[below] at (7,-0.1) {7};
\end{tikzpicture}\\

&\begin{tikzpicture}
\draw[semithick, ->] (0,10/9) -- (0,0);
\node[left] at (0,5/9) {conj};
\node[right] at (0,5/9) {$\tau_4$};
\end{tikzpicture}&\\

&$\tau_4\tau_7\tau_6\tau_5\tau_3\tau_2\tau_1$&\\
$=$&$\tau_4\tau_7\tau_6\tau_3\tau_2\tau_1\tau_5$&\begin{tikzpicture}[scale=8/9]
\draw[thick, <-] (1.15,0) -- (1.85,0);
\draw[thick, <-] (2.15,0) -- (2.85,0);
\draw[thick, <-] (3.15,0) -- (3.85,0);
\draw[thick, ->] (4.15,0) -- (4.85,0);
\draw[thick, <-] (5.15,0) -- (5.85,0);
\draw[thick, <-] (6.15,0) -- (6.85,0);
\draw[fill] (1,0) circle [radius=0.07];
\draw[fill] (2,0) circle [radius=0.07];
\draw[fill] (3,0) circle [radius=0.07];
\draw[fill] (4,0) circle [radius=0.07];
\draw[fill] (5,0) circle [radius=0.07];
\draw[fill] (6,0) circle [radius=0.07];
\draw[fill] (7,0) circle [radius=0.07];
\node[below] at (1,-0.1) {1};
\node[below] at (2,-0.1) {2};
\node[below] at (3,-0.1) {3};
\node[below] at (4,-0.1) {4};
\node[below] at (5,-0.1) {5};
\node[below] at (6,-0.1) {6};
\node[below] at (7,-0.1) {7};
\end{tikzpicture}\\

&\begin{tikzpicture}
\draw[semithick, ->] (0,10/9) -- (0,0);
\node[left] at (0,5/9) {conj};
\node[right] at (0,5/9) {$\tau_5$};
\end{tikzpicture}&\\

&$\tau_5\tau_4\tau_7\tau_6\tau_3\tau_2\tau_1$&\\
$=$&$\tau_5\tau_4\tau_7\tau_3\tau_2\tau_1\tau_6$&\begin{tikzpicture}[scale=8/9]
\draw[thick, <-] (1.15,0) -- (1.85,0);
\draw[thick, <-] (2.15,0) -- (2.85,0);
\draw[thick, <-] (3.15,0) -- (3.85,0);
\draw[thick, <-] (4.15,0) -- (4.85,0);
\draw[thick, ->] (5.15,0) -- (5.85,0);
\draw[thick, <-] (6.15,0) -- (6.85,0);
\draw[fill] (1,0) circle [radius=0.07];
\draw[fill] (2,0) circle [radius=0.07];
\draw[fill] (3,0) circle [radius=0.07];
\draw[fill] (4,0) circle [radius=0.07];
\draw[fill] (5,0) circle [radius=0.07];
\draw[fill] (6,0) circle [radius=0.07];
\draw[fill] (7,0) circle [radius=0.07];
\node[below] at (1,-0.1) {1};
\node[below] at (2,-0.1) {2};
\node[below] at (3,-0.1) {3};
\node[below] at (4,-0.1) {4};
\node[below] at (5,-0.1) {5};
\node[below] at (6,-0.1) {6};
\node[below] at (7,-0.1) {7};
\end{tikzpicture}\\

&\begin{tikzpicture}
\draw[semithick, ->] (0,10/9) -- (0,0);
\node[left] at (0,5/9) {conj};
\node[right] at (0,5/9) {$\tau_6$};
\end{tikzpicture}&\\

&$\tau_6\tau_5\tau_4\tau_7\tau_3\tau_2\tau_1$&\\
$=$&$\tau_6\tau_5\tau_4\tau_3\tau_2\tau_1\tau_7$&\begin{tikzpicture}[scale=8/9]
\draw[thick, <-] (1.15,0) -- (1.85,0);
\draw[thick, <-] (2.15,0) -- (2.85,0);
\draw[thick, <-] (3.15,0) -- (3.85,0);
\draw[thick, <-] (4.15,0) -- (4.85,0);
\draw[thick, <-] (5.15,0) -- (5.85,0);
\draw[thick, ->] (6.15,0) -- (6.85,0);
\draw[fill] (1,0) circle [radius=0.07];
\draw[fill] (2,0) circle [radius=0.07];
\draw[fill] (3,0) circle [radius=0.07];
\draw[fill] (4,0) circle [radius=0.07];
\draw[fill] (5,0) circle [radius=0.07];
\draw[fill] (6,0) circle [radius=0.07];
\draw[fill] (7,0) circle [radius=0.07];
\node[below] at (1,-0.1) {1};
\node[below] at (2,-0.1) {2};
\node[below] at (3,-0.1) {3};
\node[below] at (4,-0.1) {4};
\node[below] at (5,-0.1) {5};
\node[below] at (6,-0.1) {6};
\node[below] at (7,-0.1) {7};
\end{tikzpicture}\\

&\begin{tikzpicture}
\draw[semithick, ->] (0,10/9) -- (0,0);
\node[left] at (0,5/9) {conj};
\node[right] at (0,5/9) {$\tau_7$};
\end{tikzpicture}&\\

&$\tau_7\tau_6\tau_5\tau_4\tau_3\tau_2\tau_1$&\begin{tikzpicture}[scale=8/9]
\draw[thick, <-] (1.15,0) -- (1.85,0);
\draw[thick, <-] (2.15,0) -- (2.85,0);
\draw[thick, <-] (3.15,0) -- (3.85,0);
\draw[thick, <-] (4.15,0) -- (4.85,0);
\draw[thick, <-] (5.15,0) -- (5.85,0);
\draw[thick, <-] (6.15,0) -- (6.85,0);
\draw[fill] (1,0) circle [radius=0.07];
\draw[fill] (2,0) circle [radius=0.07];
\draw[fill] (3,0) circle [radius=0.07];
\draw[fill] (4,0) circle [radius=0.07];
\draw[fill] (5,0) circle [radius=0.07];
\draw[fill] (6,0) circle [radius=0.07];
\draw[fill] (7,0) circle [radius=0.07];
\node[below] at (1,-0.1) {1};
\node[below] at (2,-0.1) {2};
\node[below] at (3,-0.1) {3};
\node[below] at (4,-0.1) {4};
\node[below] at (5,-0.1) {5};
\node[below] at (6,-0.1) {6};
\node[below] at (7,-0.1) {7};
\end{tikzpicture}
\end{tabular}
\end{center}
\caption{A demonstration of how to write $\tau_7\tau_6\tau_5\tau_4\tau_3\tau_2\tau_1$ as a conjugation of $\tau_3\tau_4\tau_2\tau_6\tau_7\tau_5\tau_1$, with the corresponding orientations of $\calp_7$ at every step.}
\label{fig:3426751}
\end{figure}



We are now ready for the proof of Theorem~\ref{thm:preserve-Coxeter}.
\begin{proof}[Proof of Theorem~\ref{thm:preserve-Coxeter}]
As $w$ and $w'$ differ by a sequence of admissible conjugations, it suffices to consider the case where $w'$ differs from $w$ by a single conjugation.  Let $w\in\calt_n$ be a Coxeter element, $\tau_k$ initial or final in $w$, and $w'=\tau_k w \tau_k$.

We have the following commutative diagram.

\begin{center}
\begin{tikzpicture}
\node at (0,1.8) {$S$};
\node at (0,0) {$\tau_k(S)$};
\node at (3.4,1.8) {$w(S)$};
\node at (3.4,0) {$\tau_k(w(S))$};
\draw[semithick, ->] (0,1.3) -- (0,0.5);
\node[left] at (0,0.9) {$\tau_k$};
\draw[semithick, ->] (0.5,0) -- (2.5,0);
\node[below] at (1.5,0) {$w'$};
\draw[semithick, ->] (0.5,1.8) -- (2.5,1.8);
\node[above] at (1.5,1.8) {$w$};
\draw[semithick, ->] (3.4,1.3) -- (3.4,0.5);
\node[right] at (3.4,0.9) {$\tau_k$};
\end{tikzpicture}
\end{center}

Write a $w$-orbit $\eso=(S^0,S^1,\dots,S^{m-1})$, where as with the orbit board, $S^i=w(S^{i-1})$. Consider the superscripts to be mod $m$, the size of the orbit, so $S^0=S^m=w(S^{m-1})$.  Then, via the commutative diagram above, $\eso'=(\tau_k(S^0),\tau_k(S^1),\dots,\tau_k(S^{m-1}))$ is a $w'$-orbit.  This creates a bijection between $w$-orbits and $w'$-orbits that preserves the orbit sizes.

When $j\not=k$, $\chi_j(S^i)=\chi_j(\tau_k(S^i))$ as $\tau_k$ can only change the value of $\chi_k$, so $$\sum\limits_{S\in\eso}\chi_j(S)=\sum\limits_{S'\in\eso'}\chi_j(S').$$

If $\tau_k$ is final in $w$, then it is the first toggle applied when performing $w$.  Also $\tau_k$ only appears once when applying $w$, since $w$ is a Coxeter element.  Thus, in this scenario $\chi_k(\tau_k(S^i))=\chi_k(S^{i+1})$. So $$\sum\limits_{S'\in\eso'}\chi_k(S')=\chi_k(S^1)+\chi_k(S^2)+\cdots+\chi_k(S^{m-1})+\chi_k(S^0)=\sum\limits_{S\in\eso}\chi_k(S).$$

Analogously, if $\tau_k$ is initial in $w$, then $\chi_k(\tau_k(S^j))=\chi_k(S^{j-1})$.  Thus, $$\sum\limits_{S'\in\eso'}\chi_k(S')=\chi_k(S^{m-1})+\chi_k(S^0)+\chi_k(S^1)\cdots+\chi_k(S^{m-2})=\sum\limits_{S\in\eso}\chi_k(S).$$

Therefore, the sum of any $\chi_j$ is the same in $\eso$ and $\eso'$. As these orbits have the same size, the average of any $\chi_j$ is also the same across them. See Example~\ref{ex:shift-orbs}.
\end{proof}

\begin{remark}
In general, conjugation does not preserve homomesy of statistics under the respective maps.  Theorem~\ref{thm:preserve-Coxeter} describes only a specific scenario in which homomesy is preserved, namely when the conjugation is admissible (as is the case for any two Coxeter elements) and the statistic is a linear combination of the indicator functions $\chi_j$.  If we were to find other types of homomesic statistics for orbits under one Coxeter element, they would not necessarily remain homomesic under a different Coxeter element.
\end{remark}


\begin{ex}\label{ex:shift-orbs}
Figure~\ref{fig:shift-orbs} contains an orbit under the action $w=\tau_3\tau_4\tau_2\tau_6\tau_7\tau_5\tau_1$ starting with $1010010$ and an orbit under the action of $\varphi$ starting with $1010100$.  Notice that in the sequence of admissible conjugations to go from $w$ to $\varphi$, we conjugate by $\tau_4$ once, $\tau_5$ and $\tau_6$ each twice, and $\tau_7$ thrice, and each conjugation is by a final toggle, as described in Example~\ref{ex:conjugation}.  As in the proof of Theorem~\ref{thm:preserve-Coxeter}, in the $w$-orbit board, if we slide columns 4, 5, 6, and 7 up by 1 row, 2 rows, 2 rows, and 3 rows respectively, we get the $\varphi$-orbit board.
\end{ex}

\begin{remark}
Notice that while conjugation in the toggle group preserves the corresponding orbit
structures (total number of orbits and multiset of orbit sizes) and that a sequence of
admissible conjugations preserves the homomesic property of any statistic that is a linear
combination of the $\chi_j$ statistics, many other propositions we have made along the way
do not hold for generic Coxeter elements.  In particular, any statement about which
independent sets are in a given orbit does not extend to general Coxeter elements.  In the
orbit on the left in Figure~\ref{fig:shift-orbs}, notice that parts 1, 2, and 4 of
Lemma~\ref{lem:21} are violated, though part 3 holds for orbits under any Coxeter element by
the same proof.  Also notice that the orbit contains four symmetrical independent sets, but
contains independent sets whose reverses are not also in the orbit.  This shows that
Proposition~\ref{prop:atmost2symm} does not hold for arbitrary Coxeter elements.  In fact,
when $n$ is odd, it can be shown that any given orbit under the map
$\tau_2\tau_4\cdots\tau_{n-1}\tau_1\tau_3\tau_5\cdots\tau_n$ always either consists entirely
of symmetrical independent sets or contains no symmetrical independent sets.  This is because $\tau_1\tau_3\tau_5\cdots\tau_n$ consists of entirely commuting toggles, as does $\tau_2\tau_4\cdots\tau_{n-1}$.  So by symmetry $\tau_1$ acts as $\tau_n$, and $\tau_3$ acts as $\tau_{n-2}$, and so on.
\end{remark}

\begin{figure}
\begin{center}\phantom{1}\hfill
\begin{tabular}{c||c|c|c|c|c|c|c}
&\textbf{1}&\textbf{2}&\textbf{3}&\textbf{4}&\textbf{5}&\textbf{6}&\textbf{7}\\\hline\hline\\[-1.111em]
$S$&1&0&1&0&0&1&0\\\hline\\[-1.111em]
$w(S)$&0&0&0&0&0&0&0\\\hline\\[-1.111em]
$w^2(S)$&1&0&1&0&1&0&1\\\hline\\[-1.111em]
$w^3(S)$&0&0&0&0&0&1&0\\\hline\\[-1.111em]
$w^4(S)$&1&0&0&1&0&0&0\\\hline\\[-1.111em]
$w^5(S)$&0&1&0&0&0&0&1\\\hline\\[-1.111em]
$w^6(S)$&0&0&1&0&1&0&0\\\hline\\[-1.111em]
$w^7(S)$&1&0&0&0&0&0&1\\\hline\\[-1.111em]
$w^8(S)$&0&1&0&0&1&0&0\\\hline\\[-1.111em]
$w^9(S)$&0&0&0&1&0&0&1
\end{tabular}\hfill\begin{tabular}{c||c|c|c|c|c|c|c}
&\textbf{1}&\textbf{2}&\textbf{3}&\textbf{4}&\textbf{5}&\textbf{6}&\textbf{7}\\\hline\hline
$S'$&1&0&1&0&1&0&0\\\hline\\[-1.111em]
$\varphi(S')$&0&0&0&0&0&1&0\\\hline\\[-1.111em]
$\varphi^2(S')$&1&0&1&0&0&0&1\\\hline\\[-1.111em]
$\varphi^3(S')$&0&0&0&1&0&0&0\\\hline\\[-1.111em]
$\varphi^4(S')$&1&0&0&0&1&0&1\\\hline\\[-1.111em]
$\varphi^5(S')$&0&1&0&0&0&0&0\\\hline\\[-1.111em]
$\varphi^6(S')$&0&0&1&0&1&0&1\\\hline\\[-1.111em]
$\varphi^7(S')$&1&0&0&0&0&0&0\\\hline\\[-1.111em]
$\varphi^8(S')$&0&1&0&1&0&1&0\\\hline\\[-1.111em]
$\varphi^9(S')$&0&0&0&0&0&0&1
\end{tabular}\hfill\phantom{1}\end{center}
\caption{Left: the orbit under the action of $w=\tau_3\tau_4\tau_2\tau_6\tau_7\tau_5\tau_1$ containing \hbox{$S=1010010$}. Right: the orbit under the action of $\varphi$ containing \hbox{$S'=1010100$}.  See Example~\ref{ex:shift-orbs}.}
\label{fig:shift-orbs}
\end{figure}

\section{Enumerating Independent Sets and $\varphi$-Orbits}\label{sec:counting}

In this section we present enumerative formulas for the numbers of $\varphi $-orbits and reversible
$\varphi$-orbits of $\cali_{n}$.  Numerical data and the Online Encyclopedia of Integer Sequences led us to a
conjectured formula for the number of $\varphi$-orbits and connected them with
binary necklaces and bracelets.  This also helped point us
towards the snake-partition of orbit boards used in the proof of Propp's original
conjecture in Section~\ref{sec:togglemain}.  Our main tools are Burnside's Lemma and some bijections.

The first proposition is well known and is easy to see directly from the defining recursion.

\begin{proposition}\label{fib1}
The number of independent sets in $\cali_n$ is $F_{n+2}$, the $(n+2)^\text{nd}$ Fibonacci
number, with $F_0=0$, $F_1=1$ and $F_n=F_{n-1}+F_{n-2}$ for $n\geq 2$.
\end{proposition}


\begin{proposition}\label{fib2}
Let $k\in\zz^+$.\\(a) The number of symmetrical independent sets in $\cali_{2k}$ is $F_{k+1}$.\\
(b) The number of symmetrical independent sets in $\cali_{2k-1}$ is $F_{k+2}$.
\end{proposition}

\begin{proof}
(a) By symmetry, $k\in S$ if and only if $k+1\in S$. This means neither $k$ nor $k+1$ can be in $S$, since $S$ is independent. Therefore, the first $k-1$ vertices can be any independent set, and the last $k-1$ vertices is its reverse.  Thus, the symmetrical sets in $\cali_{2k}$ are in bijection with $\cali_{k-1}$, so the number of them is $F_{k+1}$.\\
(b) If the middle vertex $k$ is not in the set, then the first $k-1$ vertices can be any
independent set, and the last $k-1$ vertices is its reverse.  Therefore, there are
$\#\cali_{k-1}=F_{k+1}$ symmetrical independent sets of this type.  If vertex $k$ is in the
set, then neither $k-1$ nor $k+1$ can be.  Then the first $k-2$ vertices can be any
independent set, and the last $k-2$ vertices is its reverse.  There are $\#\cali_{k-2}=F_k$
symmetrical independent sets of this type.  Thus, the total number of symmetrical sets in
$\cali_{2k-1}$ is $F_k+F_{k+1}=F_{k+2}$.
\end{proof}

To count the number of $\varphi$-orbits, there is a connection with binary necklaces, which we now introduce.


\begin{defn}
As with compositions, two binary strings of length $n$ are said to be \textbf{cyclically
equivalent} if one is a cyclic rotation of the other.  Otherwise, the strings are
\textbf{cyclically inequivalent}.
A \textbf{binary necklace} is an equivalence class of binary strings under cyclic
equivalence (i.e., under the action of the cyclic group $C_{n}$).
A \textbf{binary bracelet} of length $n$ is an equivalence class of length $n$ binary
strings under the action of the dihedral group $D_n$  generated by cyclic rotation and reversal.
The \textbf{length} of a
binary necklace or bracelet is the length of the any string in the equivalence class.
\end{defn}

\begin{ex}
There are six binary necklaces of length 4:\\
\{0000\},\\
\{1000,0100,0010,0001\},\\
\{1100,1001,0011,0110\},\\
\{1010,0101\},\\
\{1110,1101,1011,0111\}, and\\
\{1111\},\\
all of which are also binary bracelets. For $n\leq 5$, we get the same equivalence classes under $C_n$ and $D_n$.
For $n=6$, there are 14 binary necklaces, but only 13 binary bracelets, with the
(distinct) $C_{6}$-classes of 101100 and of
001101, which are reversals of one another, combining into a single class under the
$D_{6}$-action~\cite[Seq.~A000029, A000031]{oeis}.
\end{ex}


Now we use Burnside's Lemma~\cite[Lemma 7.24.5]{ec2}, to count length $n$ binary necklaces and
bracelets with no subsequence $11$.  For this we first need to count binary strings that
will contain no subsequence 11 even when the group action makes the first and last elements
adjacent.

\begin{lemma}\label{binstringsno11}
The number of binary strings of length $n$ with no subsequence $11$ that do not both start
and end with 1 is $F_{n-1}+F_{n+1}$, for $n\geq1$.
\end{lemma}

\begin{proof}
This formula can be easily confirmed for $n=1$ and $n=2$, so assume $n\geq3$.  There are two
types of binary strings $s$:

\textbf{Case 1:}  If $s$ begins with 1, then it has the form $10\_\_\_\_\_\_\_\_\_0$, where
the blank space represents an independent set of $\calp_{n-3}$.  So there are $F_{n-1}$ strings of this form.

\textbf{Case 2:} If $s$ begins with 0, then it has the form $0\_\_\_\_\_\_\_\_\_\_\_$, where
the blank space represents an independent set of $\calp_{n-1}$.  So there are $F_{n+1}$
strings of this form.
\end{proof}

\begin{lemma}[Burnside's Lemma]
Let $Y$ be a finite set and $G$ a subgroup of $\ss_Y$.  For each $w\in G$, let
$\fix(w)=\{y\in Y: w(y)=y\}$ be the set of elements of $Y$ fixed by $w$.  Let $Y/G$ be the
set of orbits of $G$.  Then $$\#(Y/G)=\frac{1}{\#G}\sum\limits_{w\in G}\#\fix(w).$$
\end{lemma}

\begin{prop}\label{countnecklaces}
The number of binary necklaces of length $n$ with no subsequence 11 is given by~\cite[Seq.~A000358]{oeis}:
\begin{equation}\label{eq:necklaces}
\frac{1}{n}\sum\limits_{d\mid n}\phi(n/d)(F_{d-1}+F_{d+1}),
\end{equation}
where $\phi$ represents Euler's totient function.
\end{prop}

\begin{proof}
In the context of Burnside's Lemma, $G$ is the cyclic group $C_n$ of order $n$ and $Y$ is
the set of binary strings of length $n$ with no subsequence 11 that also do not both start
and end with 11.  For any $d|n$, the number of elements of $Y$ fixed by a group element of
order $n/d$ in $C_{n}$ is $F_{d-1}+F_{d+1}$ by Lemma~\ref{binstringsno11}.  By elementary
group theory, there are $\phi (n/d)$ elements of order $n/d$, and the result follows.
\end{proof}

\begin{prop}
The number of binary bracelets of length $n$ with no subsequence 11
is given by
\begin{equation}\label{eq:bracelets}
\frac{1}{2}\left(F_{\left\lfloor n/2\right\rfloor+2}+\frac{1}{n}\sum\limits_{d\mid
n}\phi(n/d)(F_{d-1}+F_{d+1})\right).
\end{equation}
This sequence appears as~\cite[Seq.~A129526]{oeis}.
\end{prop}

\begin{proof}
The cases of $n\leq4$ can be computed case by case, so we assume $n\geq5$.  As above, $Y$ is the set of binary strings of length $n$ with no
subsequence 11 that also do not both start and end with 1, but now $G = D_{n}$.
By Burnside's Lemma, the number of
binary bracelets with no subsequence 11 is \begin{eqnarray*}\frac{1}{2n}\sum\limits_{w\in
D_n}\#\fix(w)&=&\frac{1}{2n}\left(\sum\limits_{w\text{ reflection in }
D_n}\#\fix(w)+\sum\limits_{w\text{ rotation in }
D_n}\#\fix(w)\right)\\&=&\frac{1}{2n}\left(\sum\limits_{w\text{ reflection in }
D_n}\#\fix(w)+\sum\limits_{d\mid n}\phi(n/d)(F_{d-1}+F_{d+1})\right).\end{eqnarray*}
The question now is what is fixed by a reflection?  We split this into two cases.

\textbf{Case 1: $\mathbf{n}$ odd.}
Here each reflection leaves exactly one point fixed.  Without loss of generality, assume that we are working with the
reflection that reverses a string, which fixes exactly the palindromic
strings.  Since our binary strings cannot both start and end with 1, any palindromic string
both begins and ends with 0.

Similar to the proof of Lemma \ref{binstringsno11}, it is easy to show that there
are $F_{(n+1)/2}$ palindromic strings with 0 in the center, and $F_{(n-1)/2}$ with 1 in the
center. Thus, there are $F_{(n+1)/2}+F_{(n-1)/2}=F_{(n+3)/2}$ strings fixed by a
reflection.

Therefore, in the odd $n$ case, the number of binary bracelets with no subsequence 11
is $$\frac{1}{2n}\left(nF_{(n+3)/2}+\sum\limits_{d\mid n}\phi(n/d)(F_{d-1}+F_{d+1})\right)=
\frac{1}{2}\left(F_{(n+3)/2}+\frac{1}{n}\sum\limits_{d\mid n}\phi(n/d)(F_{d-1}+F_{d+1})\right).$$

\textbf{Case 2: $\mathbf{n}$ even.} Here there are two types of reflections of a regular
$n$-gon, those that fix no vertices and those that fix two vertices.


For the $n/2$ reflections that fix no vertices, we assume without loss of generality that we
are working with the reflection that reverses a string, which fixes exactly
the palindromic strings.  It is straightforward to show there are $F_{n/2}$ such strings.



For the $n/2$ reflections that fix two vertices, we assume without loss of generality that we are working with the reflection that fixes the digits in positions 1 and $n/2+1$.  There are four possibilities, because the digits in positions 1 and $n/2+1$ can each be 0 or 1. 
It is easy to show that there are $F_{n/2+1}$ such strings with 0 in both fixed positions, $F_{n/2-1}$ with 1 in both fixed positions, and $2F_{n/2}$ with different values in the fixed positions.





Adding together the four possibilities, there are \begin{eqnarray*}F_{n/2-1}+F_{n/2}+F_{n/2}+F_{n/2+1}&=&F_{n/2+1}+F_{n/2}+F_{n/2+1}\\
&=&F_{n/2+1}+F_{n/2+2}\end{eqnarray*} such strings.

Therefore, in the even $n$ case, the number of binary bracelets with no subsequence 11 is \begin{eqnarray*}&&\frac{1}{2n}\left(\frac{n}{2}\left(F_{n/2}+F_{n/2+1}+F_{n/2+2}\right)+\sum\limits_{d\mid n}\phi(n/d)(F_{d-1}+F_{d+1})\right)\\&=&
\frac{1}{2}\left(F_{n/2+2}+\frac{1}{n}\sum\limits_{d\mid n}\phi(n/d)(F_{d-1}+F_{d+1})\right).\end{eqnarray*}

In both the odd and even cases, the formula holds.
\end{proof}

\begin{cor}\label{numreversible}
The number of length $n$ binary necklaces with no subsequence 11 which are the same under reversal
is $F_{\lfloor n/2\rfloor+2}$.
\end{cor}

\begin{proof}
Let $S$ be this number, and $D$ the number of such necklaces which are not equal to their
reversals.  Then $S+D=(\ref{eq:necklaces} )$ and $S+\frac{1}{2}D=(\ref{eq:bracelets}
)$ (equations above).  Eliminating $D$ gives the result.
\end{proof}
\begin{ex}
There are 10 binary necklaces of length 9 without the subsequence 11. One representative
from each of them is as follows: 000000000, 100000000, 101000000, 100100000, 100010000,
101010000, 100100100, 101010100, 101001000, 000100101. For each of the last two, the reverse
is not a cyclic rotation of itself, and thus not the same necklace. These are the necklaces
that do not count in Corollary~\ref{numreversible}.  There are $8=F_6$ other necklaces.
\end{ex}

\begin{thm}\label{thm:countcomps}
There is a bijection between the set of binary necklaces of length $n$ with no subsequence
11, and cyclically inequivalent compositions of $n$ with each part equal to 1 or 2.
Therefore, by Proposition~\ref{countnecklaces}, the number of cyclically inequivalent
compositions of $n$ with each part equal to 1 or 2 is $\frac{1}{n}\sum\limits_{d\mid
n}\varphi(n/d)(F_{d-1}+F_{d+1})$~\cite[Seq.~A000358]{oeis}.
\end{thm}

\begin{proof}
Take a binary necklace of length $n$ with no subsequence 11, and write it in the form of a
string $s$ so the first character is 0, which exists because there is no subsequence 11.
Then create a composition of $n$ into parts 1 and 2, by replacing each occurrence of 01 in
$s$ with 2, and each occurrence of 0 in $s$ not followed by 1 with a 1.  For example
\begin{center}\begin{tabular}{ccc}010010001&$\longleftrightarrow$&212112.
\end{tabular}
\end{center}
To show that this bijection is well-defined, if we cyclically rotate the composition to the left, we get a cyclically equivalent composition, but this corresponds to rotating the string $s$ to the left once or twice depending on if the first part of the composition is 1 or 2.  Therefore, the new string is part of the same necklace.  An example of this is shown below.  The strings in the necklace that begin with 1 have no corresponding composition.\begin{center}\begin{tabular}{ccc}010010001&$\longleftrightarrow$&212112\\
100100010&$\longleftrightarrow$&\\
001000101&$\longleftrightarrow$&121122\\
010001010&$\longleftrightarrow$&211221\\
100010100&$\longleftrightarrow$&\\
000101001&$\longleftrightarrow$&112212\\
001010010&$\longleftrightarrow$&122121\\
010100100&$\longleftrightarrow$&221211\\
101001000&$\longleftrightarrow$&
  \end{tabular}\end{center}\vspace{-0.27 in}
\end{proof}

\begin{cor}\label{cor:countcomps}
The number of cyclically inequivalent compositions $c$ of $n$ with each part equal to 1 or 2, for which the reverse order of $c$ is cyclically equivalent to $c$ is $F_{\lfloor n/2\rfloor+2}$.
\end{cor}

\begin{proof}
This follows easily from Corollary~\ref{numreversible} and the proof of Theorem~\ref{thm:countcomps}.
\end{proof}

\begin{theorem}\label{thm:countorbits}
The total number of $\varphi$-orbits of $\cali_n$ is
$\frac{1}{n-1}\sum\limits_{d|(n-1)}\phi((n-1)/d)(F_{d-1}+F_{d+1})$.  The total number of
reversible orbits is $F_{\lceil n/2\rceil+1}$.
\end{theorem}

\begin{proof}
Combine the bijection of Proposition~\ref{prop:snakebijection} with
Theorem~\ref{thm:countcomps} for orbits and with Corollary~\ref{cor:countcomps},
for reversible orbits. Note that $\lfloor(n-1)/2\rfloor+2=\lceil n/2\rceil+1$.
\end{proof}

A more user-friendly version of the latter formula is:
\begin{itemize}\item The number of reversible orbits in $\cali_{2k-1}$ is $F_{k+1}$.\item The number of reversible orbits in $\cali_{2k}$ is $F_{k+1}$.\end{itemize}

\section{Sizes of Orbits}\label{sec:orbitsizes}
In the initial data gathering, the authors observed several mysterious patterns in the sizes
of $\varphi$-orbits. For example, almost all of the orbits had size congruent to $1-n\bmod 4$ and certain orbit sizes like 4 and 6 never appeared at all (for any $n$).
These mysteries were also cleared up via the snake compositions of
Section~\ref{sec:togglemain}, which leads to a simple characterization
(Proposition~\ref{prop:orbitsizes}) of the orbit size corresponding to a given composition.  From
this we derive nontrivial consequences for the existence of orbits of a fixed sizes as $n$
varies, summarized for small sizes in Figure~\ref{fig:table}.

The basic idea uses
Theorem~\ref{thm:snakeorbit}: whenever the snake composition $c$ contains a 2, the next snake
starts two positions down, and when a snake composition contains a 1, the next snake starts
three positions down.  So the size of the orbit should be $|c'|$, where $c'$ is obtained
from $c$ by replacing each 1 with a 3.  For example, if $c=221121$, then $c'=223323$ is a
composition of 15, which is the size of the orbit in Example~\ref{board}.


This naive approach can fail, however, in the case where periodicity of $c$ leads one to
create a ``super-orbit'' rather than an orbit, e.g.,
the orbit of $\cali_7$ in Figure~\ref{fig:2121} given by snake composition $c=2121$. Here
$|c'| = 2+3+2+3=10$, but
$S^5=S^0$ so the orbit actually has size 5, and the board repeats itself.  Therefore,
given a snake composition such as 2121 made up entirely of a repeated segment (here
21), we must divide by the number of times the smallest repeated segment repeats itself
(here 2) to get the correct orbit size.

\begin{figure}
\begin{center}\begin{tabular}{c||c|c|c|c|c|c|c}
&\textbf{1}&\textbf{2}&\textbf{3}&\textbf{4}&\textbf{5}&\textbf{6}&\textbf{7}\\\hline\hline\\[-1.111em]
$S^0$&{\cellcolor{red}{\color{white}\textbf{1}}}&\textbf{0}&{\cellcolor{red}{\color{white}\textbf{1}}}&\textbf{0}&\textbf{0}&\textbf{0}&\textbf{0}\\\hline\\[-1.111em]
$S^1$&\textbf{0}&\textbf{0}&\textbf{0}&{\cellcolor{red}{\color{white}\textbf{1}}}&\textbf{0}&{\cellcolor{red}{\color{white}\textbf{1}}}&\textbf{0}\\\hline\\[-1.111em]
$S^2$&{\cellcolor{purple}{\color{white}\textbf{1}}}&\textbf{0}&\textbf{0}&\textbf{0}&\textbf{0}&\textbf{0}&{\cellcolor{red}{\color{white}\textbf{1}}}\\\hline\\[-1.111em]
$S^3$&\textbf{0}&{\cellcolor{purple}{\color{white}\textbf{1}}}&\textbf{0}&{\cellcolor{purple}{\color{white}\textbf{1}}}&\textbf{0}&\textbf{0}&\textbf{0}\\\hline\\[-1.111em]
$S^4$&\textbf{0}&\textbf{0}&\textbf{0}&\textbf{0}&{\cellcolor{purple}{\color{white}\textbf{1}}}&\textbf{0}&{\cellcolor{purple}{\color{white}\textbf{1}}}\\\hline\\[-1.111em]
$S^5$&{\cellcolor{green}{\color{white}\textbf{1}}}&\textbf{0}&{\cellcolor{green}{\color{white}\textbf{1}}}&\textbf{0}&\textbf{0}&\textbf{0}&\textbf{0}\\\hline\\[-1.111em]
$S^6$&\textbf{0}&\textbf{0}&\textbf{0}&{\cellcolor{green}{\color{white}\textbf{1}}}&\textbf{0}&{\cellcolor{green}{\color{white}\textbf{1}}}&\textbf{0}\\\hline\\[-1.111em]
$S^7$&{\cellcolor{blue}{\color{white}\textbf{1}}}&\textbf{0}&\textbf{0}&\textbf{0}&\textbf{0}&\textbf{0}&{\cellcolor{green}{\color{white}\textbf{1}}}\\\hline\\[-1.111em]
$S^8$&\textbf{0}&{\cellcolor{blue}{\color{white}\textbf{1}}}&\textbf{0}&{\cellcolor{blue}{\color{white}\textbf{1}}}&\textbf{0}&\textbf{0}&\textbf{0}\\\hline\\[-1.111em]
$S^9$&\textbf{0}&\textbf{0}&\textbf{0}&\textbf{0}&{\cellcolor{blue}{\color{white}\textbf{1}}}&\textbf{0}&{\cellcolor{blue}{\color{white}\textbf{1}}}
\end{tabular}\end{center}
\caption{}
\label{fig:2121}
\end{figure}

\begin{defn}
Call a composition $c$ \textbf{periodic} if it consists entirely of adjacent copies of the
same repeated substring. Given a composition $c$, let $\psi(c)$ be the number of times the
smallest repeated segment repeats itself to make up $c$. If $\psi (c)=1$, we call $c$
\textbf{aperiodic}.
\end{defn}

\begin{ex}
The composition $21221$ is aperiodic, while $\psi(22122212)=2$ and $\psi(222)=3$.
\end{ex}

\begin{prop}\label{prop:orbitsizes}
Given a $\varphi$-orbit $\eso$ containing snake composition $c$, let $N_1(c)$ be the number of occurrences of 1 in $c$ and $N_2(c)$ be the number of occurrences of 2 in $c$.  Then the size of the orbit $\eso$ is $\frac{3N_1(c)+2N_2(c)}{\psi(c)}$.
\end{prop}

Therefore, given any orbit size, we can characterize exactly for which $n$ there is an orbit of $\cali_n$ with that size, and how many such orbits.

\begin{prop}\label{orbit2}
There is an orbit of $\cali_n$ of size 2 exactly when $n$ is odd.  In this case, the orbit is unique.
\end{prop}

\begin{proof}
The only composition of 2 into parts 2 and 3 is the composition 2.  By Proposition~\ref{prop:snakebijection},
an orbit has
size 2 if and only if a snake composition corresponding to the orbit is of the form
$222\cdots2$, with $k$ 2s repeated in $\cali_{2k+1}$.
\end{proof}

\begin{ex}\label{ex:orbit-size2}
For $n$ odd, the two independent sets in a $\varphi$-orbit of $\cali_n$ of size 2 are the
empty set and the set $\{1,3,5,\dots,n\}$, as shown in the following orbit board in
$\cali_9$. 

\begin{center}\begin{tabular}{c||c|c|c|c|c|c|c|c|c}
&\textbf{1}&\textbf{2}&\textbf{3}&\textbf{4}&\textbf{5}&\textbf{6}&\textbf{7}&\textbf{8}&\textbf{9}\\\hline\hline\\[-1.111em]
$S^0$ & \textbf{0} & \textbf{0} & \textbf{0} & \textbf{0} & \textbf{0}  & \textbf{0}  & \textbf{0}  & \textbf{0}  & \textbf{0} \\\hline\\[-1.111em]
$S^1$&{\cellcolor{red}{\color{white}\textbf{1}}} & \textbf{0} &{\cellcolor{red}{\color{white}\textbf{1}}} & \textbf{0} &{\cellcolor{red}{\color{white}\textbf{1}}} & \textbf{0} &{\cellcolor{red}{\color{white}\textbf{1}}} & \textbf{0} &{\cellcolor{red}{\color{white}\textbf{1}}}
\end{tabular}\end{center}
\end{ex}

\begin{prop}\label{orbit3}
For any $n\geq2$, there is a unique orbit of $\cali_n$ of size 3.
\end{prop}

\begin{proof}
The only composition of 3 into parts 2 and 3 is the composition 3.  Therefore, a
$\varphi$-orbit has size 3 if and only if the corresponding snake composition has the form
$111\cdots1$, with $k$ ones repeated in $\cali_{k+1}$.
\end{proof}

\begin{ex}
The three independent sets $S^0$, $S^1$, $S^2$ in the orbit of size 3 are the sets of all elements of $[n]$ congruent to $0,1,2\bmod 3$, respectively. An example of the orbit board for this orbit of $\cali_7$ is shown below.
\begin{center}\begin{tabular}{c||c|c|c|c|c|c|c}
&\textbf{1}&\textbf{2}&\textbf{3}&\textbf{4}&\textbf{5}&\textbf{6}&\textbf{7}\\\hline\hline\\[-1.111em]
$S^0$ & \textbf{0} & \textbf{0} & {\cellcolor{red}{\color{white}\textbf{1}}} & \textbf{0} & \textbf{0} & {\cellcolor{red}{\color{white}\textbf{1}}}& \textbf{0}\\\hline\\[-1.111em]
$S^1$ & {\cellcolor{red}{\color{white}\textbf{1}}} & \textbf{0} & \textbf{0} & {\cellcolor{red}{\color{white}\textbf{1}}} & \textbf{0} & \textbf{0} & {\cellcolor{red}{\color{white}\textbf{1}}}\\\hline\\[-1.111em]
$S^2$ & \textbf{0} & {\cellcolor{red}{\color{white}\textbf{1}}} & \textbf{0} & \textbf{0} & {\cellcolor{red}{\color{white}\textbf{1}}} & \textbf{0} & \textbf{0}
\end{tabular}\end{center}
\end{ex}

\begin{prop}\label{orbit4,6}
For every $n$, there are no orbits of $\cali_n$ of size 4 or 6.
\end{prop}

\begin{proof}
The only composition of 4 into parts 2 and 3 is the composition $2+2$.  However, this composition is periodic and therefore gives an orbit of size 2. Similarly, the only compositions of 6 into parts 2 and 3 are $2+2+2$ and $3+3$, both periodic.  These give the orbits of size 2 and 3 respectively.
\end{proof}

\begin{prop}\label{orbit5}
For $n\geq 2$, there is an orbit of $\cali_n$ of size 5 exactly when $n\equiv1\bmod3$.  In this case, the orbit is unique.
\end{prop}

\begin{proof}
The only compositions of 5 into parts 2 and 3 are $2+3$ and $3+2$, which are cyclically equivalent.  Therefore, an orbit has size 5 if and only if a snake composition corresponding to the orbit is of the form $1212\cdots12$, with $k$ total patterns of 12 repeated (thus a composition of $3k$).  This snake composition is in an orbit of $\cali_{3k+1}$.
\end{proof}

\begin{ex}
The orbit board for the $\varphi$-orbit of $\cali_7$ of size 5 is below.
\begin{center}\begin{tabular}{c||c|c|c|c|c|c|c}
&\textbf{1}&\textbf{2}&\textbf{3}&\textbf{4}&\textbf{5}&\textbf{6}&\textbf{7}\\\hline\hline\\[-1.111em]
$S^0$& {\cellcolor{red}{\color{white}\textbf{1}}} & \textbf{0} & {\cellcolor{red}{\color{white}\textbf{1}}} & \textbf{0}  & \textbf{0}  & \textbf{0}  & \textbf{0} \\\hline\\[-1.111em]
$S^1$ & \textbf{0}  & \textbf{0}  & \textbf{0} & {\cellcolor{red}{\color{white}\textbf{1}}} & \textbf{0} & {\cellcolor{red}{\color{white}\textbf{1}}} & \textbf{0} \\\hline\\[-1.111em]
$S^2$& {\cellcolor{blue}{\color{white}\textbf{1}}} & \textbf{0}  & \textbf{0}  & \textbf{0}  & \textbf{0}  & \textbf{0} & {\cellcolor{red}{\color{white}\textbf{1}}}\\\hline\\[-1.111em]
$S^3$ & \textbf{0} & {\cellcolor{blue}{\color{white}\textbf{1}}} & \textbf{0} & {\cellcolor{blue}{\color{white}\textbf{1}}} & \textbf{0}  & \textbf{0}  & \textbf{0} \\\hline\\[-1.111em]
$S^4$ & \textbf{0}  & \textbf{0}  & \textbf{0}  & \textbf{0} & {\cellcolor{blue}{\color{white}\textbf{1}}} & \textbf{0} & {\cellcolor{blue}{\color{white}\textbf{1}}}
\end{tabular}\end{center}
\end{ex}

\begin{figure}[b]
\begin{center}\begin{tabular}{c|c|c|c}
&Aperiodic cyclically&&$n\geq 2$ for which\\
Orbit&inequivalent compositions &Corresponding snake&$\cali_n$ has an orbit\\
size $m$&of $m$ into parts 2 or 3&composition type&of size $m$\\\hline\hline
2&$2$&$222\cdots2$&$n\equiv1\bmod2$\\\hline
3&$3$&$111\cdots1$&all $n$\\\hline
4&none&none&none\\\hline
5&$3+2$&$1212\cdots12$&$n\equiv1\bmod3$\\\hline
6&none&none&none\\\hline
7&$3+2+2$&$122122\cdots122$&$n\equiv1\bmod5$\\\hline
8&$3+3+2$&$112112\cdots112$&$n\equiv1\bmod4$\\\hline
9&$3+2+2+2$&$12221222\cdots1222$&$n\equiv1\bmod7$\\\hline
10&$3+3+2+2$&$11221122\cdots1122$&$n\equiv1\bmod6$\\\hline
11&$3+3+3+2$&$11121112\cdots1112$&$n\equiv1\bmod5$\\
11&$3+2+2+2+2$&$1222212222\cdots12222$&$n\equiv1\bmod9$\\\hline
12&$3+2+3+2+2$&$1212212122\cdots12122$&$n\equiv1\bmod8$\\
12&$3+3+2+2+2$&$1122211222\cdots11222$&$n\equiv1\bmod8$\\
\end{tabular}\end{center}
\caption{Classification of path graphs which give certain $\varphi$-orbit sizes}
\label{fig:table}
\end{figure}

We can continue this classification for $n\geq 2$ as shown in Figure~\ref{fig:table}. For example $\cali_n$ has an orbit of size 7 if and only if
$n\equiv1\bmod5$, in which case the orbit is unique. An orbit of size
11 exists if and only if $n\equiv1\bmod5$ or $n\equiv1\bmod9$, and this orbit is unique except when
$n$ is both $1\bmod5$ and $1\bmod9$ (i.e., $n\equiv1\bmod45$) in which case there exist two
such orbits.  Also, $\cali_n$ has an orbit of size 12 if and only if $n\equiv1\bmod8$, in
which case there are always exactly two such orbits.

\begin{prop}
For even $n$, the $\varphi$-orbit of $\cali_n$ containing the empty set has size $n+1$.
\end{prop}

\begin{proof}
It is easy to see that for even $n$, $\varphi(0000\cdots00)=1010\cdots10$ and
$\varphi^2(0000\cdots00)=0000\cdots01$.  This gives the first three rows of an orbit board,
corresponding to
aperiodic snake composition $\underbrace{22\cdots2}_{(n-2)/2}1$, whose orbit size  is
$2\frac{n-2}{2}+3=n+1$.   See Figure~\ref{fig:orbit6empty} for an example of this orbit when
$n=6$.
\begin{figure}
\begin{center}\begin{tabular}{c||c|c|c|c|c|c}
&\textbf{1}&\textbf{2}&\textbf{3}&\textbf{4}&\textbf{5}&\textbf{6}\\\hline\hline\\[-1.111em]
$S^0$  & \textbf{0}   & \textbf{0}   & \textbf{0}   & \textbf{0}   & \textbf{0}   & \textbf{0} \\\hline\\[-1.111em]
$S^1$& {\cellcolor{red}{\color{white}\textbf{1}}}  & \textbf{0} & {\cellcolor{red}{\color{white}\textbf{1}}}  & \textbf{0} & {\cellcolor{red}{\color{white}\textbf{1}}}  & \textbf{0} \\\hline\\[-1.111em]
$S^2$  & \textbf{0}   & \textbf{0}   & \textbf{0}   & \textbf{0}   & \textbf{0} & {\cellcolor{red}{\color{white}\textbf{1}}}\\\hline\\[-1.111em]
$S^3$& {\cellcolor{green}{\color{white}\textbf{1}}}  & \textbf{0} & {\cellcolor{green}{\color{white}\textbf{1}}}  & \textbf{0}   & \textbf{0}   & \textbf{0} \\\hline\\[-1.111em]
$S^4$  & \textbf{0}   & \textbf{0}   & \textbf{0} & {\cellcolor{green}{\color{white}\textbf{1}}}  & \textbf{0} & {\cellcolor{green}{\color{white}\textbf{1}}}\\\hline\\[-1.111em]
$S^5$& {\cellcolor{blue}{\color{white}\textbf{1}}} & \textbf{0}   & \textbf{0}   & \textbf{0}   & \textbf{0}   & \textbf{0} \\\hline\\[-1.111em]
$S^6$  & \textbf{0} & {\cellcolor{blue}{\color{white}\textbf{1}}}  & \textbf{0} & {\cellcolor{blue}{\color{white}\textbf{1}}}  & \textbf{0} & {\cellcolor{blue}{\color{white}\textbf{1}}}
\end{tabular}\end{center}
\caption{The $\varphi$-orbit of $\cali_6$ contaning $\O$.}
\label{fig:orbit6empty}
\end{figure}
\end{proof}

\begin{thm}\label{thm:mod4}
Let $\eso$ be an $\varphi$-orbit of $\cali_n$ and $c$ be a snake composition that appears in $\eso$.  If $c$ is aperiodic, then the size of $\eso$ is congruent to $1-n\bmod4$. Furthermore, even when $c$ is periodic, the size of $\eso$ divides an integer $m\equiv1-n\bmod4$ for $m\leq 3(n-1)$ (where $m$ depends on the orbit $\eso$).
\end{thm}

\begin{proof}
Using the notation of Proposition~\ref{prop:orbitsizes}, the size of $\eso$ is
$\frac{3N_1(c)+2N_2(c)}{\psi(c)}$.  If $c=\underbrace{111\cdots1}_{n-1}$, then
$$3N_1(c)+2N_2(c)=3(n-1)\equiv1-n\bmod4.$$
Any other composition of $n-1$ into parts 1
or 2 can be formed starting with $111\cdots1$ and replacing strings of 11 with 2.  Each
time a 11 in a snake composition is changed to a 2, the sum $3N_1(c)+2N_2(c)$ is decreased
by 4.  Thus, $3N_1(c)+2N_2(c)\equiv 1-n\bmod4$ and is at most $3(n-1)$.  Thus, the size of
$\eso$ divides $3N_1(c)+2N_2(c)$ and when $c$ is aperiodic, the orbit size is given by
$3N_1(c)+2N_2(c)\equiv 1-n\bmod4$.
\end{proof}

\begin{cor}\label{cor:mod4}
For even $n$, every $\varphi$-orbit of $\cali_n$ has odd size.  Furthermore, when $n\equiv3\bmod4$, there exist no orbits with size divisible by 4.
\end{cor}

\begin{proof}
Theorem~\ref{thm:mod4} tells us that the size of any orbit $\eso$ divides an integer
$m\equiv1-n\bmod4$, which is odd for $n$ even.
Meanwhile, $n\equiv 3 \bmod 4$ forces $m \equiv 2 \bmod 4$, which $\#\eso$ must divide.
\end{proof}

\begin{cor}
When $n$ is even, any reversible $\varphi$-orbit of $\cali_n$ contains exactly one symmetrical independent set.
\end{cor}

\begin{proof}
For a reversible orbit $\eso$ and any $S\in\eso$, $S^{\rev}$ is also in $\eso$.  This partitions the independent sets in $\eso$ into pairs, with the only unpaired ones being the sets
that satisfy $S=S^{\rev}$, i.e., the symmetrical independent sets.  By
Corollary~\ref{cor:mod4}, $\eso$ has odd size so there must be an odd number of symmetrical
independent sets in $\eso$.  By Proposition~\ref{prop:atmost2symm}, an orbit can contain at
most two symmetrical independent sets, so $\eso$ contains exactly one symmetrical
independent set.
\end{proof}

\section{Connections with Order Ideals in Zigzag Posets}\label{sec:zigzagposets}

The original problem about independent sets is connected with other well-studied maps,
called \emph{promotion} and \emph{rowmotion}, on zigzag posets.  Rowmotion was introduced as
a map on antichains of a poset in~\cite{brouwer1974period}, while promotion was originally
formulated by Sch\"utzenberger in the context of standard Young
tableaux~\cite{Sch63}. Promotion and rowmotion have both been studied in various settings by
numerous authors. Here, we consider them as maps on order ideals, as discussed by
Striker and Williams~\cite{strikerwilliams}.

Our main result is an equivariant bijection between the $\varphi$-action on $\cali_{n}$ and
the rowmotion action on order ideals of a zigzag poset.  This allows us to carry over results
from the previous sections to the new context, which we now describe.

A \textbf{partially ordered set} (or \textbf{poset} for short) is a set $P$ together with a
binary relation $\leq$ that is reflexive, antisymmetric, and transitive. (A reader
unfamiliar with posets can find the necessary background in Stanley's text
~\cite[Ch.~3]{ec1ed2}.) An \textbf{order ideal} of a poset $P$ is a subset $I$ of $P$ such that if $x\in I$ and
$y\leq x$ in $P$, then $y\in I$.  The set of order ideals of $P$ is denoted $J(P)$.

\begin{definition}
The \textbf{zigzag poset} with $n$ elements, denoted $\calz_n$, is the poset consisting of
elements $a_1,...,a_n$ and relations $a_{2i-1} < a_{2i}$ and $a_{2i+1} < a_{2i}$.
(Such posets are also called \textbf{fence posets} and are discussed in \cite[p. 367]{ec1ed2}.)
\end{definition}

The zigzag posets have Hasse diagrams that can be drawn in a zigzag formation, hence the name. For example
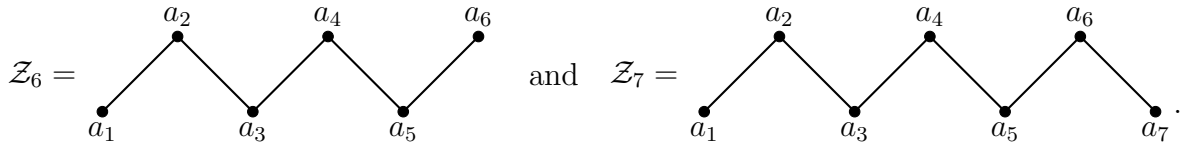
\begin{center}
\begin{tikzpicture}
\draw[thick] (0,0) -- (1,1) -- (2,0) -- (3,1) -- (4,0) -- (5,1);
\draw[fill] (0,0) circle [radius=0.07];
\draw[fill] (1,1) circle [radius=0.07];
\draw[fill] (2,0) circle [radius=0.07];
\draw[fill] (3,1) circle [radius=0.07];
\draw[fill] (4,0) circle [radius=0.07];
\draw[fill] (5,1) circle [radius=0.07];
\node[below] at (0,0) {$a_1$};
\node[above] at (1,1) {$a_2$};
\node[below] at (2,0) {$a_3$};
\node[above] at (3,1) {$a_4$};
\node[below] at (4,0) {$a_5$};
\node[above] at (5,1) {$a_6$};
\node[left] at (-0.2,0.5) {$\calz_6=$};
\node at (6,0.5) {and};
\end{tikzpicture}
\begin{tikzpicture}
\draw[thick] (0,0) -- (1,1) -- (2,0) -- (3,1) -- (4,0) -- (5,1) -- (6,0);
\draw[fill] (0,0) circle [radius=0.07];
\draw[fill] (1,1) circle [radius=0.07];
\draw[fill] (2,0) circle [radius=0.07];
\draw[fill] (3,1) circle [radius=0.07];
\draw[fill] (4,0) circle [radius=0.07];
\draw[fill] (5,1) circle [radius=0.07];
\draw[fill] (6,0) circle [radius=0.07];
\node[below] at (0,0) {$a_1$};
\node[above] at (1,1) {$a_2$};
\node[below] at (2,0) {$a_3$};
\node[above] at (3,1) {$a_4$};
\node[below] at (4,0) {$a_5$};
\node[above] at (5,1) {$a_6$};
\node[below] at (6,0) {$a_7$};
\node[left] at (-0.2,0.5) {$\calz_7=$};
\node[left] at (6.5,0) {.};
\end{tikzpicture}
\end{center}

The toggle group on $J(\calz_n)$ is defined analogously to that of $\cali_n$.  The main
conceptual difference is that while an element can always be toggled out of an independent
set, the same is not true for all elements in an order ideal of $P$.

\begin{defn}The \textbf{toggle} $t_i:J(\calz_n)\ra J(\calz_n)$ is defined
as $$t_i(I):=\left\{\begin{array}{ll}I\cup\{a_i\}&\text{if }a_i\not\in I
\text{ and }
I\cup\{a_i\} \in
J(\calz_n)\\
I\sm\{a_i\}&\text{if }a_i\in I
\text{ and }
I\sm\{a_i\} \in
J(\calz_n)\\
I&\text{otherwise}
\end{array}\right.$$ The \textbf{toggle group} of $J(\calz_n)$, denoted $\tog(\calz_n)$, is
generated by the toggles $t_i$ for $i\in[n]$.
\end{defn}

As with toggles on $\cali_n$, it is clear that $t_i^2=1$ for any $i\in [n]$.  The following is an analogue of Proposition~\ref{prop:togglescommute}.

\begin{prop}\label{zigzagcommute}
Two toggles $t_i,t_j\in\tog(\calz_n)$ commute if and only if $|i-j|\not=1$.
\end{prop}

\begin{proof}
When $|i-j|>1$, $a_i$ and $a_j$ are incomparable (i.e., neither $a_i\leq a_j$ nor $a_j\leq a_i$) so inclusion of one of $a_i$ or $a_j$ cannot affect whether the other can be included.
\end{proof}

Two special elements of $\tog(\calz_n)$ are \textbf{promotion} $\pro=t_n\cdots t_2 t_1$ and
\textbf{rowmotion} $$\row=\left\{\begin{array}{ll}t_{n-1}t_{n-3}\cdots
t_3t_1t_nt_{n-2}\cdots t_2&\text{if }n\text{ is even}\\t_{n}t_{n-2}\cdots
t_3t_1t_{n-1}t_{n-3}\cdots t_2&\text{if }n\text{ is odd}\end{array}\right..$$

These maps have been studied on general posets by many authors. Rowmotion can be defined
for any poset, and promotion can be defined for any poset which can be embedded as a
rowed-and-columned poset (such as $\calz_n$)~\cite{strikerwilliams}.

Next we define the equivariant bijection $\eta$ that takes us between independent sets of
$\calp_n$ and order ideals of $\calz_n$.

\begin{prop}
The map $\eta:\cali_n\rightarrow J(\calz_n)$ defined below is a bijection:
$$\eta(S):=\{a_i \;|\; i\in[n], i \text{ odd, } i\not\in S\} \cup \{a_i \;|\; i\in[n], i \text { even, } i\in S\}.$$	
\end{prop}

\begin{proof}
Note that for $I\subseteq\calz_n$ to be an order ideal, it must satisfy the property that whenever $a_{2j}$ is in $I$, so are $a_{2j-1}$ and $a_{2j+1}$ (or just $a_{2j-1}$ if $2j=n$).

Suppose $a_{2j}\in \eta(S)$. Then $2j\in S$. Since $S$ is independent, $2j-1$ and $2j+1$ are not in $S$. So $a_{2j-1}\in\eta(S)$ and (when $2j\not=n$) $a_{2j+1}\in\eta(S)$. Hence $\eta(S)\in J(\calz_n)$.  The inverse of $\eta$ is given by
$$\eta^{-1}(I)=\{i\;|\; i\in[n], i \text{ odd, } a_i\not\in I\} \cup \{i\;|\; i\in[n], i \text{ even, } a_i\in I\}$$ which always produces an independent set for $I\in J(\calz_n)$ by analogous reasoning. Thus $\eta$ is a bijection.
\end{proof}

\begin{example} Let $n=7$ and $S=1001010=\{1,4,6\}$. Then $a_1\not\in\eta(S)$ and $a_3,a_5,a_7\in\eta(S)$ because $1\in S$ and $3,5,7\not\in S$.  Also, $a_2\not\in\eta(S)$ and $a_4,a_6\in\eta(S)$, since $2\not\in S$ and $4,6\in S$.  This correspondence is shown below, where the hollow circles are included in $\calz_7$ but not the order ideal $\eta(S)$.
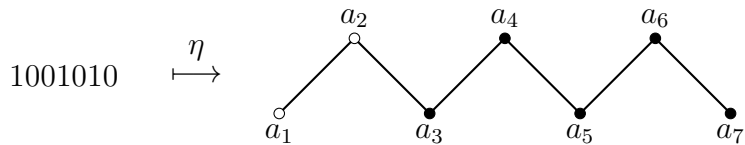
\begin{center}
\begin{tikzpicture}
\draw[thick] (0.045,0.045) -- (0.955,0.955);
\draw[thick] (1.045, 0.955) -- (2,0) -- (3,1) -- (4,0) -- (5,1) -- (6,0);
\draw (0,0) circle [radius=0.07];
\draw (1,1) circle [radius=0.07];
\draw[fill] (2,0) circle [radius=0.07];
\draw[fill] (3,1) circle [radius=0.07];
\draw[fill] (4,0) circle [radius=0.07];
\draw[fill] (5,1) circle [radius=0.07];
\draw[fill] (6,0) circle [radius=0.07];
\node[below] at (0,0) {$a_1$};
\node[above] at (1,1) {$a_2$};
\node[below] at (2,0) {$a_3$};
\node[above] at (3,1) {$a_4$};
\node[below] at (4,0) {$a_5$};
\node[above] at (5,1) {$a_6$};
\node[below] at (6,0) {$a_7$};
\node[above] at (-1.1,0.567) {$\eta$};
\node at (-1.1,0.5) {$\longmapsto$};
\node[left] at (-2,0.5) {1001010};
\end{tikzpicture}
\end{center}
\end{example}


\begin{prop}\label{correspond}
For every $i\in[n]$, $\eta\circ\tau_i=t_i\circ\eta$.  Thus,
$\eta\circ\varphi=\pro\circ\eta$, making $\eta$ an \textbf{equivariant} bijection, as shown in the
following commutative diagrams.
\begin{center}
\begin{tikzpicture}
\begin{scope}
\node at (0,1.8) {$\cali_n$};
\node at (0,0) {$\cali_n$};
\node at (3.25,1.8) {$J(\calz_n)$};
\node at (3.25,0) {$J(\calz_n)$};
\draw[semithick, ->] (0,1.3) -- (0,0.5);
\node[left] at (0,0.9) {$\tau_i$};
\draw[semithick, ->] (0.5,0) -- (2.5,0);
\node[below] at (1.5,0) {$\eta$};
\draw[semithick, ->] (0.5,1.8) -- (2.5,1.8);
\node[above] at (1.5,1.8) {$\eta$};
\draw[semithick, ->] (3.25,1.3) -- (3.25,0.5);
\node[right] at (3.25,0.9) {$t_i$};
\end{scope}
\begin{scope}[shift={(6.3,0)}]
\node at (0,1.8) {$\cali_n$};
\node at (0,0) {$\cali_n$};
\node at (3.25,1.8) {$J(\calz_n)$};
\node at (3.25,0) {$J(\calz_n)$};
\draw[semithick, ->] (0,1.3) -- (0,0.5);
\node[left] at (0,0.9) {$\varphi$};
\draw[semithick, ->] (0.5,0) -- (2.5,0);
\node[below] at (1.5,0) {$\eta$};
\draw[semithick, ->] (0.5,1.8) -- (2.5,1.8);
\node[above] at (1.5,1.8) {$\eta$};
\draw[semithick, ->] (3.25,1.3) -- (3.25,0.5);
\node[right] at (3.25,0.9) {$\pro$};
\end{scope}
\end{tikzpicture}
\end{center}
\end{prop}



As with $\calt_n$, $\tog(\calz_n)$ is the quotient of a Coxeter group, and it is clearly isomorphic to $\calt_n$ via Proposition~\ref{correspond}. We define Coxeter elements analogously to $\calt_n$, and $\pro$ and $\row$ are two examples of Coxeter elements in $\tog(\calz_n)$.

\begin{defn}
A \textbf{Coxeter element} in $\tog(\calz_n)$ is a product of each of the $n$ toggles $t_1,t_2,\dots,t_n$ each used exactly once, in some order.
\end{defn}

\begin{theorem}Any two Coxeter elements in $\tog(\calz_n)$ are conjugate.  In particular, promotion and rowmotion are conjugate.\end{theorem}

In \cite[\S5]{strikerwilliams}, Striker and Williams prove that $\pro$ and $\row$ are conjugate for any rowed-and-columned poset.  However, in the case of $\tog(\calz_n)$, the conjugacy of any two Coxeter elements follows from the conjugacy of Coxeter elements in $\calt_n$ and Proposition~\ref{correspond}. Thus, the orbit structure of $J(\calz_n)$ under $\row$, or any other Coxeter element in $\tog(\calz_n)$, is the same as that of $\pro$, which by Proposition~\ref{correspond}, is the same as the orbit structure of $\varphi$ on $\cali_n$.



Using Proposition~\ref{correspond}, we can restate Corollary~\ref{cor:hom-Coxeter} for toggling in $J(\calz_n)$, as follows.
\begin{cor}
Let $w$ be a Coxeter element in $\tog(\calz_n)$.  Let $\chi_{a_j}:J(Z_n)\ra\{0,1\}$ be the indicator function of $a_j$. Then on $w$-orbits in $J(Z_n)$, the following statistics are homomesic.\begin{itemize}
\item If $n$ is odd, then $\chi_{a_j}-\chi_{a_{n+1-j}}$ is 0-mesic for every $j\in[n]$. Also $2\chi_{a_1}-\chi_{a_2}$ and $2\chi_{a_n}-\chi_{a_{n-1}}$ are both 1-mesic.
\item If $n$ is even, then $\chi_{a_j}+\chi_{a_{n+1-j}}$ is 1-mesic for every $j\in[n]$. Also $2\chi_{a_1}-\chi_{a_2}$ is 1-mesic and $2\chi_{a_n}-\chi_{a_{n-1}}$ is 0-mesic.
\end{itemize}
\end{cor}

\begin{proof}
From the definition of $\eta$, it is clear that for any $S\in\cali_n$, $$\chi_{a_j}(\eta(S))=\left\{\begin{array}{ll}
\chi_j(S) &\text{if }j\text{ is even}\\
1-\chi_j(S)
&\text{if }j\text{ is odd}\end{array}\right..$$ The rest of the proof follows by restating Corollary~\ref{cor:hom-Coxeter} into the language of $J(\calz_n)$ via Proposition~\ref{correspond}.
\end{proof}

Notice that our statements above are significantly more complicated to state, forcing us to
divide into odd and even cases.  This would also make direct proofs of them in the
$J(\calz_{n})$-setting more unwieldy.  It is much easier to handle them via translation to
the $\cali_{n}$ context.  This shows the efficacy of Striker's notion of generalized
toggling.

It is well-known and not hard to see that for any graded poset $P$ of rank $r$, there is a
rowmotion orbit on $J(P)$ of size $r+1$ generated by the empty ideal, where $\row^{i}(\O
)$ consists of all elements of rank $\leq i-1$.  In particular, $J(\calz_{n})$ has a
rowmotion orbit of size 3. It is not directly obvious that this is the only orbit of this size.
But since the orbit
structure of $\row$ is the same as that of $\varphi$ on $\cali_n$, uniqueness follows
from  Proposition~\ref{orbit3}.




In other proven examples of homomesy for rowmotion on posets, the map generally has a small
order and a cyclic sieving phenomenon has been found.  However, the rowmotion map on
$J(\calz_n)$ has a large order, and thus a natural cyclic sieving result is unlikely, which
makes the homomesy for this poset particularly interesting.

\bibliography{bibliography}

\newcommand{\etalchar}[1]{$^{#1}$}
\begin{thebibliography}{EFGJMPR}

\bibitem[AST13]{ast}
D.~Armstrong, C.~Stump, and H.~Thomas.
\newblock A uniform bijection between nonnesting and noncrossing partitions.
\newblock {\em Transactions of the American Mathematical Society},
  365(8):4121--4151, 2013.
\newblock Also available as \arxiv{1101.1277v2}.

\bibitem[BB05]{bjornerbrenti}
A.~Bj{\"o}rner and F.~Brenti.
\newblock {\em Combinatorics of {C}oxeter groups}.
\newblock Springer Verlag, New York, 2005.

\bibitem[BS74]{brouwer1974period}
A.~Brouwer and L.~Schrijver.
\newblock On the period of an operator, defined on antichains.
\newblock {\em Stichting Mathematisch Centrum. Zuivere Wiskunde}, (ZW
  24/74):1--13, 1974.

\bibitem[CF95]{cameronfonder}
P.~Cameron and D.~Fon{-}der{-}Flaass.
\newblock Orbits of antichains revisited.
\newblock {\em European J. Combin.}, 16(6):545--554, 1995.

\bibitem[EE09]{eriksson2009conjugacy}
H.~Eriksson and K.~Eriksson.
\newblock Conjugacy of {C}oxeter elements.
\newblock {\em Electron. J. Combin.}, 16(2):\#R4, 2009.

\bibitem[EFGJMPR]{efgjmpr}
D.~Einstein, M.~Farber, E.~Gunawan, M.~Joseph, M.~Macauley, J.~Propp, and
  S.~Rubinstein-Salzedo.
\newblock Noncrossing partitions, toggles, and homomesies.
\newblock {\em Electron. J. Combin.}, 23(3), 2016.
\newblock Also available at \arxiv{1510.06362v2}.

\bibitem[EP13]{einpropp}
D.~Einstein and J.~Propp.
\newblock Combinatorial, piecewise-linear, and birational homomesy for products
  of two chains.
\newblock {\em \arxiv{1310.5294}}, 2013.

\bibitem[Had16]{shahrzad}
S.~Haddadan.
\newblock Some instances of homomesy among ideals of posets.
\newblock {\em \arxiv{1410.4819v3}}, 2016.

\bibitem[Jos17]{antichain-toggling}
M.~Joseph.
\newblock Antichain toggling and rowmotion.
\newblock {\em \arxiv{1709.09331}}, 2017.

\bibitem[Pan09]{panyushev}
D.~Panyushev.
\newblock On orbits of antichains of positive roots.
\newblock {\em European J. Combin.}, 30(2):586--594, 2009.
\newblock Also available at \arxiv{0711.3353v2}.

\bibitem[Pet15]{peterseneulerian}
T.~K. Petersen.
\newblock {\em Eulerian Numbers}.
\newblock Springer, New York, 2015.

\bibitem[PR15]{propproby}
J.~Propp and T.~Roby.
\newblock Homomesy in products of two chains.
\newblock {\em Electron. J. Combin.}, 22(3), 2015.
\newblock Also available at \arxiv{1310.5201v5}.

\bibitem[Rob16]{robydac}
T.~Roby.
\newblock Dynamical algebraic combinatorics and the homomesy phenomenon.
\newblock In {\em Recent Trends in Combinatorics}, pages 619--652. Springer,
  2016.
\newblock Also available at
  \url{http://www.math.uconn.edu/~troby/homomesyIMA2015Revised.pdf}.

\bibitem[RSW04]{csp}
V.~Reiner, D.~Stanton, and D.~White.
\newblock The cyclic sieving phenomenon.
\newblock {\em Journal of Combinatorial Theory, Series A}, 108(1):17--50, 2004.

\bibitem[S{\etalchar{+}}16]{sage}
W.\thinspace{}A. Stein et~al.
\newblock {\em {S}age {M}athematics {S}oftware ({V}ersion 7.3)}.
\newblock The Sage Development Team, 2016.
\newblock \url{http://www.sagemath.org}.

\bibitem[Sag11]{cspsagan}
B.~Sagan.
\newblock The cyclic sieving phenomenon: A survey.
\newblock {\em London Math. Soc. Lecture Note Ser}, 392:183--233, 2011.
\newblock Also available at \arxiv{1008.0790v3}.

\bibitem[Sch63]{Sch63}
M.~P. Sch{\"u}tzenberger.
\newblock Quelques remarques sur une construction de {S}chensted.
\newblock {\em Math. Scand.}, 12:117--128, 1963.

\bibitem[Shi97]{shi1997enumeration}
J.~Shi.
\newblock The enumeration of {C}oxeter elements.
\newblock {\em J. Algebraic Combin.}, 6(2):161--171, 1997.

\bibitem[Slo16]{oeis}
N.~J.~A. Sloane.
\newblock {\em OEIS, Online Encyclopedia of Integer Sequences}, 2016.
\newblock \url{https://oeis.org/}.

\bibitem[Sta99]{ec2}
R.~Stanley.
\newblock {\em Enumerative combinatorics, volume 2}.
\newblock Cambridge University Press, 1999.

\bibitem[Sta11]{ec1ed2}
R.~Stanley.
\newblock {\em Enumerative combinatorics, volume 1, 2nd edition}.
\newblock Cambridge University Press, 2011.
\newblock Also available at \url{http://math.mit.edu/~rstan/ec/ec1/}.

\bibitem[Str15]{strikerRS}
J.~Striker.
\newblock The toggle group, homomesy, and the {R}azumov-{S}troganov
  correspondence.
\newblock {\em Electron. J. Combin.}, 22(2):P2--57, 2015.
\newblock Also available at \arxiv{1503.08898v1}.

\bibitem[Str16]{strikergentog}
J.~Striker.
\newblock Rowmotion and generalized toggle groups.
\newblock {\em arXiv:1601.03710}, 2016.

\bibitem[SW12]{strikerwilliams}
J.~Striker and N.~Williams.
\newblock Promotion and rowmotion.
\newblock {\em European J. Combin.}, 33:1919--1942, 2012.
\newblock Also available at \arxiv{1108.1172v3}.

\end{thebibliography}
\bibliographystyle{halpha}

\end{document}